\documentclass{birkjour}
\pdfoutput=1
 \newtheorem{thm}{Theorem}[section]

 \newtheorem{prop}[thm]{Proposition}
 \theoremstyle{definition}
 \newtheorem{defn}[thm]{Definition}
 \theoremstyle{remark}

 \numberwithin{equation}{section}

\usepackage{algorithm}
\usepackage[noend]{algpseudocode}
\usepackage{euscript}
\def\cal{\EuScript}
\usepackage{url}
\usepackage{bbm}

\def\BA{{\bf A}}
 \def\Be{{\bf e}}

\def\BJ{{\bf J}} 
\def\CJ{\mbox{\boldmath$\cal J$}}
\def\BP{{\bf P}} 

\mathcode`\@="8000
{\catcode`\@=\active \gdef@{\mkern1mu}} 
\newcommand{\A}{{\mathcal A}}
\newcommand{\ess}{{\mathrm{ess}}}
\newcommand{\N}{{\mathbbm{N}}}
\newcommand{\R}{{\mathbbm{R}}}
\newcommand{\Z}{{\mathbbm{Z}}}
\newcommand{\tr}{{\mathrm{tr}}}
\DeclareMathOperator*{\slim}{s-lim}
\def\a{{\sf a}}
\def\b{{\sf b}}
\def\c{{\sf c}}
\def\d{{\sf d}}
\def\bcd{{\rm dim}_{B}}
\def\eps{\varepsilon}

\begin{document}

%

\title[Spectral Approximation for Quasiperiodic Jacobi Operators]{Spectral Approximation for Quasiperiodic\\ Jacobi Operators}

\author{Charles Puelz}

\address{%
Department of Computational and Applied Mathematics\\
Rice University\\
6100 Main Street -- MS 134\\
Houston, Texas\ \ 77005--1892\\
USA
}

\email{cpuelz@rice.edu}


\author{Mark Embree}
\address{%
Department of Mathematics\\
Virginia Tech\\
225 Stanger Street -- 0123\\
Blacksburg, Virginia\ \ 24061\\
USA
}
\email{embree@vt.edu}

\author{Jake Fillman}
\address{%
Department of Mathematics\\
Rice University\\
6100 Main Street -- MS 136\\
Houston, Texas\ \ 77005--1892\\
USA
}
\email{jdf3@rice.edu}


\begin{abstract}
Quasiperiodic Jacobi operators arise as mathematical models of quasicrystals 
and in more general studies of structures exhibiting aperiodic order.  
The spectra of these self--adjoint operators can be quite exotic, such as Cantor sets,
and their fine properties yield insight into the associated quantum dynamics, 
that is, the one--parameter unitary group 
that solves the time--dependent Schr\"odinger equation.
Quasiperiodic operators can be approximated by periodic ones, the spectra of which
can be computed via two finite dimensional eigenvalue problems.  Since long periods 
are necessary for detailed approximations, both computational efficiency and 
numerical accuracy become a concern. We describe a simple method for 
numerically computing the spectrum of a period--$K$ Jacobi operator in $O(K^2)$ operations, 
then use the algorithm to investigate the spectra of Schr\"odinger operators 
with Fibonacci, period doubling, and Thue--Morse potentials.
\end{abstract}

\maketitle


\section{Introduction}
For given sets of parameters $\{a_n\}, \{b_n\} \in \ell^{\infty}(\Z)$, 
the corresponding \emph{Jacobi operator} $\CJ : \ell^2(\Z) \rightarrow\ell^2(\Z)$ is defined entrywise by
\begin{equation}
\label{eq0}
(\CJ \psi)_n = a_{n-1} \psi_{n-1} + b_n \psi_n + a_{n} \psi_{n+1} 
\end{equation}
for all $\{\psi_n \} \in \ell^2(\Z)$.  
When there exists a positive integer $K$ such that
\[
a_n = a_{n+K}, \quad b_n = b_{n+K} \rlap{\qquad for all $n\in\Z$,}
\] 
the Jacobi operator is said to be \emph{periodic} with \emph{period}~$K$.
Here we are interested in computing the spectrum $\sigma(\CJ)$
of such an operator when the period $K$ is very long,
as a route to high fidelity numerical approximations 
of the more intricate spectra of operators with aperiodic coefficients.
These spectra are important, as they can help us understand the quantum dynamics of 
solutions to the time-dependent Schr\"odinger equation~\cite{Las96}.

The spectrum of a period-$K$ Jacobi operator can be calculated from classical Floquet--Bloch theory, the relevant highlights of which we briefly recapitulate for later reference. (Our presentation most closely follows Toda~\cite[Ch.~4]{Tod89}; see also, e.g., \cite[Ch.~5]{Sim11}, \cite[Ch.~7]{Tes00},\cite{vMo76}.) For a scalar~$E$, any solution to $\CJ\psi = E \psi$ satisfies
\begin{equation} \label{eq:MKp}
\begin{bmatrix} \psi_{pK+1} \\ \psi_{pK} \end{bmatrix} 
= 
M_K^p \begin{bmatrix} \psi_{1} \\ \psi_{0} \end{bmatrix}
\end{equation}
for each $p \in \Z$, where $M_K \equiv M_K(E)$ denotes the $2\times 2$ \emph{monodromy matrix}
\begin{equation} \label{eq:monod}
 M_K = 
   \begin{bmatrix} {E-b_K\over a_K} & -{a_{K-1}\over a_K} \\[.5em] 1 & 0 \end{bmatrix}
   \cdots
   \begin{bmatrix} {E-b_2\over a_2} & -{a_{1}\over a_2} \\[.5em] 1 & 0 \end{bmatrix}
   \begin{bmatrix} {E-b_1\over a_1} & -{a_{K}\over a_1} \\[.5em] 1 & 0 \end{bmatrix}.
\end{equation}
Note that 
\begin{equation} \label{eq:detMK}
 \det(M_K) = {a_{K-1} \over a_K} \cdots {a_1\over a_2} {a_K\over a_1} = 1.
\end{equation}
Now when $\CJ$ is periodic, $E\in\sigma(\CJ)$ provided $\CJ\psi = E\psi$ with bounded nontrivial $\psi=\{\psi_n\}$,%
\footnote{The spectrum of a Jacobi operator $\CJ$ is given by the closure of the set of $E \in \R$ for which a nontrivial polynomially bounded solution to $\CJ \psi = E \psi$ exists. When $\CJ$ is periodic, one either has a bounded solution or all solutions grow exponentially on at least one half-line.} 
which by~(\ref{eq:MKp}) and (\ref{eq:detMK}) requires
the eigenvalues of $M_K$ to have unit modulus. 
From
\begin{equation} \label{eq:charpoly}
\det(\gamma - M_K) =  \gamma^2 - \tr(M_K) \gamma + 1,\end{equation}
where $\tr(\cdot)$ denotes the trace, the eigenvalues of $M_K$ have
unit modulus when 
\begin{equation} \label{eq:disc}
-2 \le \tr(M_K) \le 2.
\end{equation}
Since $\tr(M_k)$ is a degree-$K$ polynomial in $E$, in principle we can find $\sigma(\CJ)$ by solving the two polynomial equations $\tr(M_K(E)) = \pm 2$, giving $\sigma(\CJ)$ as the union of $K$ real intervals. For large~$K$ such numerical computations can incur significant errors, as illustrated in~\cite[Sec.~7.1]{DEG12}.  
Alternatively, note that if $[\psi_1, \psi_0]^T$ is an eigenvector of $M_K$ associated with
eigenvalue $\gamma=e^{i \theta}$, then 
by  definition~(\ref{eq0}) and periodicity, $\CJ\psi = E \psi$ and $\psi_{K+j}=\gamma \kern1pt \psi_j$ imply
\[ 
 \begin{bmatrix}
   b_1 & a_1 &  &  & e^{-i\theta} a_K \\
   a_1 & b_2 & \ddots &  &  \\
     & \ddots & \ddots & \ddots &  \\
     &  & \ddots & b_{K-1} & a_{K-1} \\
   e^{i \theta}  a_K &  &  & a_{K-1} & b_K
 \end{bmatrix}
     \begin{bmatrix} \psi_{1} \\[.3em] \psi_2 \\[.3em] \vdots \\[.3em] \psi_{K-1} \\[.3em] \psi_{K} \end{bmatrix}
  = E \begin{bmatrix} \psi_{1} \\[.3em] \psi_2 \\[.3em] \vdots \\[.3em] \psi_{K-1} \\[.3em] \psi_{K} \end{bmatrix},
\]
where all unspecified entries in the matrix equal zero.
Solving this $K\times K$ symmetric matrix eigenvalue problem
for any $\theta\in[0,2\pi)$ gives $K$ points in $\sigma(\CJ)$, one in each of the $K$ intervals.
We can compute $\sigma(\CJ)$ directly by noting that $\theta=0$ and $\theta=\pi$ give
the endpoints of these intervals.  We shall thus focus on the two $K\times K$ symmetric matrices
\begin{equation}
\label{eq:Jpm}
 \mathbf{J}_\pm =
 \begin{bmatrix}
   b_1 & a_1 &  &  & \pm a_K \\
   a_1 & b_2 & \ddots &  &  \\
     & \ddots & \ddots & \ddots &  \\
     &  & \ddots & b_{K-1} & a_{K-1} \\
   \pm  a_K &  &  & a_{K-1} & b_K
 \end{bmatrix}.
 \end{equation}
More precisely, enumerating the $2K$ eigenvalues of $\mathbf{J}_+$ and $\BJ_-$ such that
\begin{equation} \label{eq:interleave}
E_1 < E_2 \leq E_3 < E_4 \leq \cdots \leq E_{2K-1} < E_{2K}
\end{equation}
(strict inequalities separate eigenvalues from $\BJ_+$ and $\BJ_-$), then 
\begin{equation}
\label{eq1}
\sigma(\CJ) = \bigcup_{j=1}^{K} \ [E_{2j-1}, E_{2j}].
\end{equation}
This discussion suggests three ways to compute the spectrum~(\ref{eq1}):
\begin{enumerate}
\item For a fixed~$E$, test if $E\in\sigma(\CJ)$ by explicitly calculating
      $\tr(M_k(E))$ and checking if~(\ref{eq:disc}) holds;
\item Construct the degree-$K$ polynomial $\tr(M_K(E))$ and find the roots of
      $\tr(M_K(E)) = \pm 2$;
\item Compute the eigenvalues of the two $K\times K$ symmetric matrices $\BJ_\pm$.
\end{enumerate}
The first method gives a fast way to test if $E\in\sigma(\CJ)$ for a given $E$,
but is an ineffective way to compute the entire spectrum.  
The second method suffers from the numerical
instabilities mentioned earlier.  The third approach is most favorable, but 
$O(K^3)$ complexity and numerical inaccuracies in the computed eigenvalues 
become a concern for large $K$.

Jacobi operators with long periods arise as approximations to
Schr\"o\-ding\-er operators with aperiodic potentials. 
For one special case (the almost Mathieu potential), 
Thouless~\cite{Tho83} and Lamoureux~\cite{Lam97} proposed an $O(K^2)$ algorithm
to compute the spectra of $\BJ_\pm$ for the period-$K$ approximation.
Section~\ref{sec:alg} describes a simpler $O(K^2)$ algorithm 
that does not exploit special properties of the potential, and so applies 
to any period-$K$ Jacobi operator.
Section~\ref{sec:quasi} addresses aperiodic potentials in some detail,
showing how their spectra can be covered by those of periodic approximations.  
Section~\ref{sec:exp} shows application of our algorithm to estimate the fractal dimension 
of the spectrum for aperiodic operators with potential given by primitive substitution rules.
\section{Algorithm} \label{sec:alg}
The conventional algorithm for finding all eigenvalues of a general symmetric matrix 
$\BA$ first requires the application of unitary similarity transformations to reduce
the matrix to tridiagonal form (in which all entries other than those on the main diagonal 
and the first super- and sub-diagonals are zero).  This transformation is 
the most costly part of the eigenvalue computation: for a $K\times K$
matrix, the reduction takes $O(K^3)$ operations, while the eigenvalues
of the tridiagonal matrix can be found to high precision 
in $O(K^2)$ further operations~\cite[\S8.15]{Par98}.
When many entries of $\BA$ are zero, one might exploit this structure to 
perform fewer elementary similarity transformations.
This is the case when $\BA$ is banded:
$a_{j,k} = 0$ when $|j-k|>b$, where $b$ is the \emph{bandwidth} of $\BA$.\ \ 
For fixed $b$, the tridiagonal reduction takes $O(K^2)$ operations
as $K\to\infty$.

The matrix $\BJ_\pm$ in~(\ref{eq:Jpm}) is tridiagonal plus entries in
the $(1,K)$ and~$(K,1)$ positions that give it bandwidth $b=K$.\ \  
One might still hope to somehow exploit the zero structure to 
quickly reduce $\BJ_\pm$ to tridiagonal form.
Unfortunately, the transformation that eliminates the $(1,K)$ and $(K,1)$
entries creates new nonzero entries where once there were zeros, starting a 
cascade of new nonzeros whose sequential elimination leads again to $O(K^3)$ 
complexity.  Figure~\ref{fig:spyplotsA} shows how the conventional approach
to tridiagonalizing a sparse matrix (using Givens plane rotations) eventually
creates $O(K^2)$ nonzero entries, requiring $O(K^2)$ storage and
$O(K^3)$ floating point operations.

\begin{figure}
\begin{center}
   \includegraphics[scale=0.33]{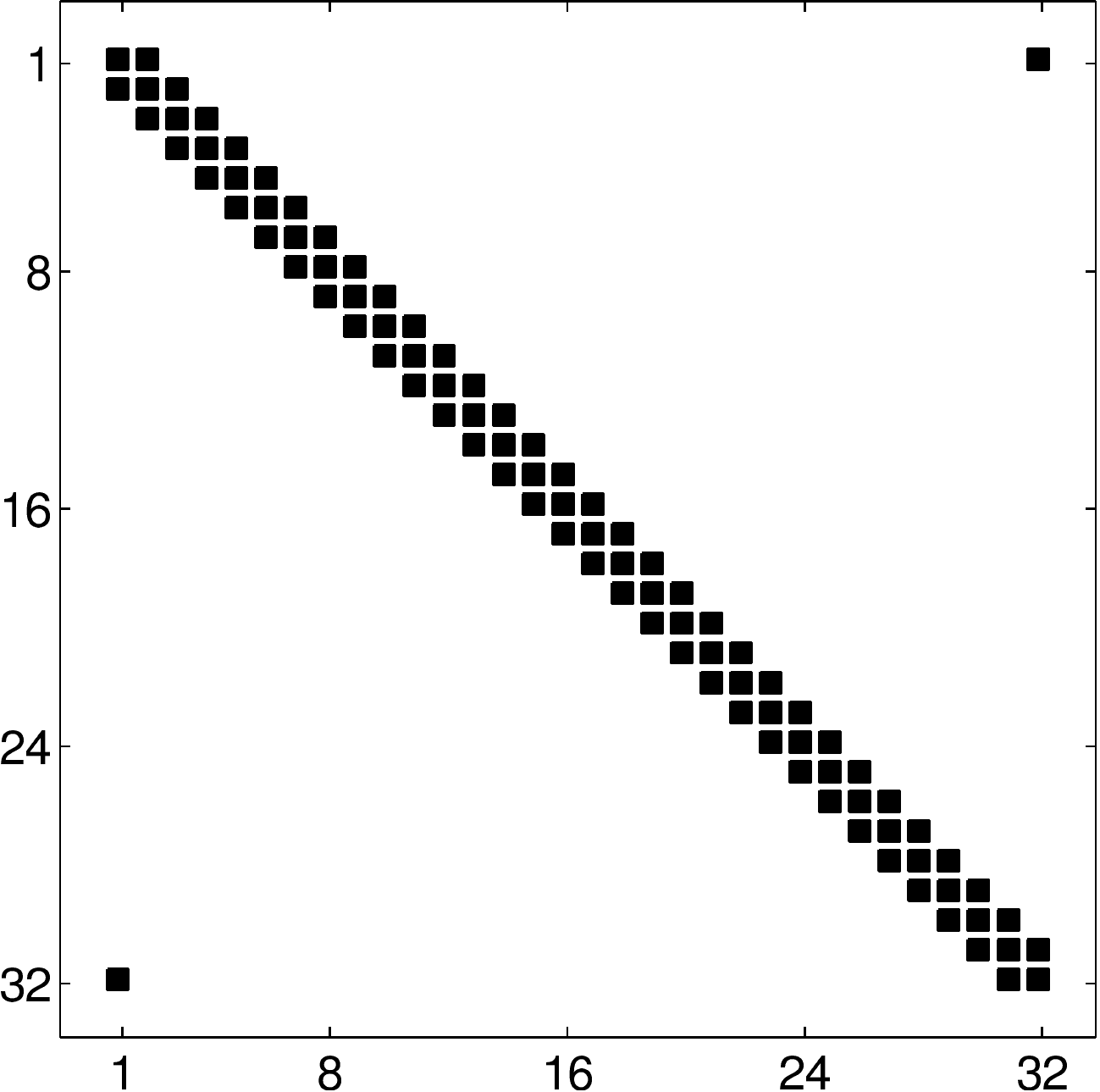} \qquad
   \includegraphics[scale=0.33]{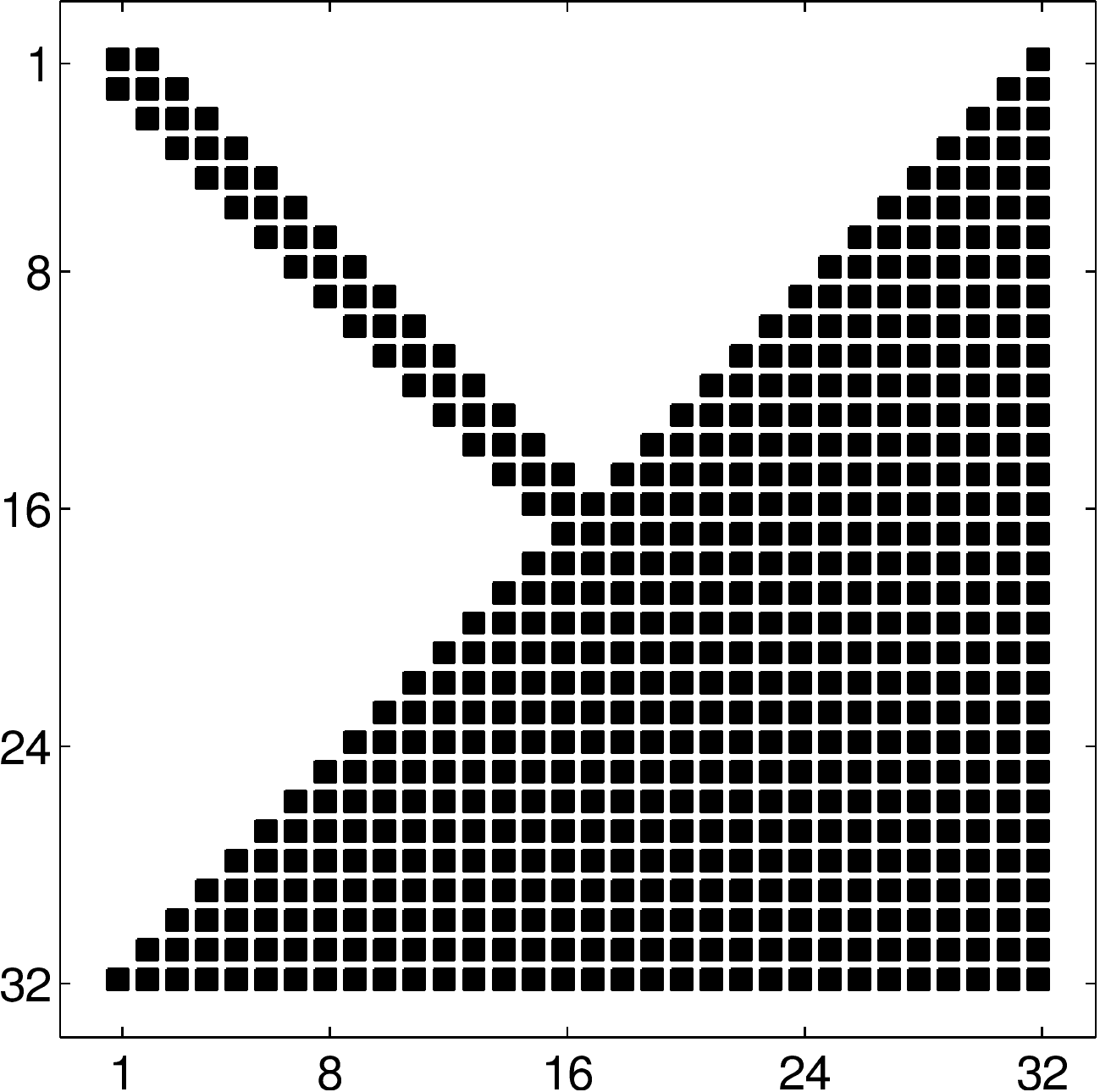}
\end{center}
\vspace*{-5pt}
\caption{\label{fig:spyplotsA}
The nonzero pattern of $\BJ_\pm$ for $K=32$ (left) and an illustration of all the
entries that are nonzero at some point in the transformation of $\BJ_\pm$ to
tridiagonal form using the conventional plane rotation approach (right).}
\vspace*{-10pt}
\end{figure}

The challenge is magnified by the eigenvalues themselves.
As described in Section~\ref{sec:quasi}, we approximate 
aperiodic operators whose spectra are Cantor sets, implying
that the eigenvalues of $\sigma(\BJ_\pm)$ will be tightly clustered.
The accuracy of these computed eigenvalues is critical to the
applications we envision (e.g., estimating the fractal dimension
of the Cantor sets), warranting use of extended (quadruple)
precision arithmetic that magnifies the cost of $O(K^3)$ operations.
Moreover, these studies often examine a family of operators over 
many parameter values (e.g., as the $\{b_n\}$ terms are scaled),
so expedient algorithms are helpful even when $K$ is tractable for a
single matrix.

Thouless~\cite{Tho83} and, in later, more complete work, Lamoureux~\cite{Lam97} 
provide an explicit unitary matrix 
$\mathbf{Q}_\pm$ such that $\mathbf{Q}^*_\pm \mathbf{J}_\pm^{} \mathbf{Q}_\pm^{}$ 
is tridiagonal in the special case of the almost Mathieu potential,
\[ a_n = 1, \qquad b_n = 2 \lambda \cos(2\pi(n \alpha + \theta)), 
\]
where $\alpha$ is a rational approximation to an irrational parameter of
true interest, and $\lambda$ and $\theta$ are constants.
This transformation gives an $O(K^2)$ algorithm for computing 
$\sigma(\BJ_\pm)$, but relies on the special structure
of the potential.

Alternatively, we note that one can compute
all eigenvalues of $\BJ_\pm$ in $O(K^2)$ time by
simply reordering the rows and columns of $\BJ_\pm$
to yield a matrix with small bandwidth that is independent of $K$.\ \ 
Recall that the nonzero pattern of a symmetric matrix $\BA$ can be 
represented as an undirected graph having vertices labeled $1, \ldots, K$,
with distinct vertices $j$ and $k$ joined by an edge if $a_{j,k} = a_{k,j} \ne 0$.
(We suppress the loops corresponding to $a_{j,j} \ne 0$.)
The graph for $\BJ_\pm$ has a single cycle,
shown in Figure~\ref{fig:reorder8} for $K=8$.\ \ 
In the conventional labeling, the corner entries in the
$(1,K)$ and $(K,1)$ positions give an edge between vertices~$1$ and~$K$.
To obtain a matrix with narrow bandwidth, simply relabel the
vertices in a breadth-first fashion starting from vertex~1, as 
illustrated in Figure~\ref{fig:reorder8}.
Now each vertex only connects
to vertices whose labels differ by at most \emph{two}:
if we permute the rows and columns of the matrix in accord with 
the relabeling, the resulting matrix will have bandwidth~2 (i.e.,
a \emph{pentadiagonal} matrix).
More explicitly, define
\[ p(j) = \left\{\begin{array}{ll} 2j-1, & j \in \{1,\ldots, \lceil K/2\rceil\};\\[.25em]
                                   2(K-j-1), & j \in \{\lceil K/2\rceil +1,\ldots, K\}.
          \end{array}\right.\]
Let $\BP = [\Be_{p(1)}, \Be_{p(2)}, \ldots, \Be_{p(K)}]$,
where $\Be_j$ is the $j$th column of the $K\times K$ identity matrix;
then $\BP\BJ_\pm\BP^*$ has bandwidth $b=2$. 
When $K=8$,
\[
\BJ_\pm = \begin{bmatrix}
                   b_1 & a_1 & & & & & & \pm a_8 \\
                   a_1 & b_2 & a_2 & & & & & \\
                   & a_2 & b_3 & a_3 & & & &  \\
                   & & a_3 & b_4 & a_4 & & &  \\
                   & & & a_4 & b_5 & a_5 & &  \\
                   & & & & a_5 & b_6 & a_6 &  \\
                   & & & & & a_6 & b_7 & a_7   \\
                   \pm a_8 & & & & & & a_7 & b_8   
          \end{bmatrix}
\]
is reordered to
\[\hspace*{-15.5pt}
\BP\BJ_\pm\BP^* = \begin{bmatrix}
                   b_1 & \pm a_8 & a_1 & & & & &  \\
                   \pm a_8 & b_8 & & a_7 & & & & \\
                   a_1 &  & b_2 &  & a_2& & &  \\
                   & a_7 &  & b_7 &  & a_6 & &  \\
                   & & a_2 & & b_3 &  & a_3 &  \\
                   & & & a_6& & b_6 &  & a_5  \\
                   & & & & a_3 &  & b_4 & a_4   \\
                   & & & & & a_5 & a_4 & b_5   
          \end{bmatrix},
\]
where unspecified entries are zero.
The tridiagonal reduction
of banded symmetric matrices has been carefully studied,
starting with Rut\-is\-hauser%
\footnote{Rutishauser was motivated by pentadiagonal matrices
that arise in the addition of continued fractions.}~\cite{Rut63};
see~\cite{BLS00} for contemporary algorithmic considerations.  
Such matrices can be reduced to tridiagonal form 
using Givens plane rotations applied in a ``bulge-chasing'' 
procedure that increases the bandwidth by one;
Figure~\ref{fig:spyplotsB} shows the nonzero entries introduced
by this reduction.%
\footnote{An alternative method that finds the eigenvalues of a 
banded symmetric matrix directly in its band form is described in~\cite[\S8.16]{Par98}.
This method, based on small Householder reflectors, is particularly effective when
a small subset of the spectrum is sought.}
Removing the $(j+2,j)$ entry introduces a new entry in the third subdiagonal
(a ``bulge'') that is ``chased'' toward the bottom right with 
$O(K)$ additional Givens rotations, each of which requires $O(1)$ floating
point operations.  
Performing this exercise for $j=1,\ldots, K-2$ amounts to $O(K^2)$
work and $O(K)$ storage
to reduce $\BP\BJ_\pm\BP^*$ to tridiagonal form, an improvement
over the usual $O(K^3)$ work and $O(K^2)$ storage.
(Our application does not require the eigenvectors of $\BJ_\pm$, so we 
do not store the transformations
that tridiagonalize $\BJ_\pm$.)

\begin{figure}
\begin{center}
   \includegraphics[scale=0.38]{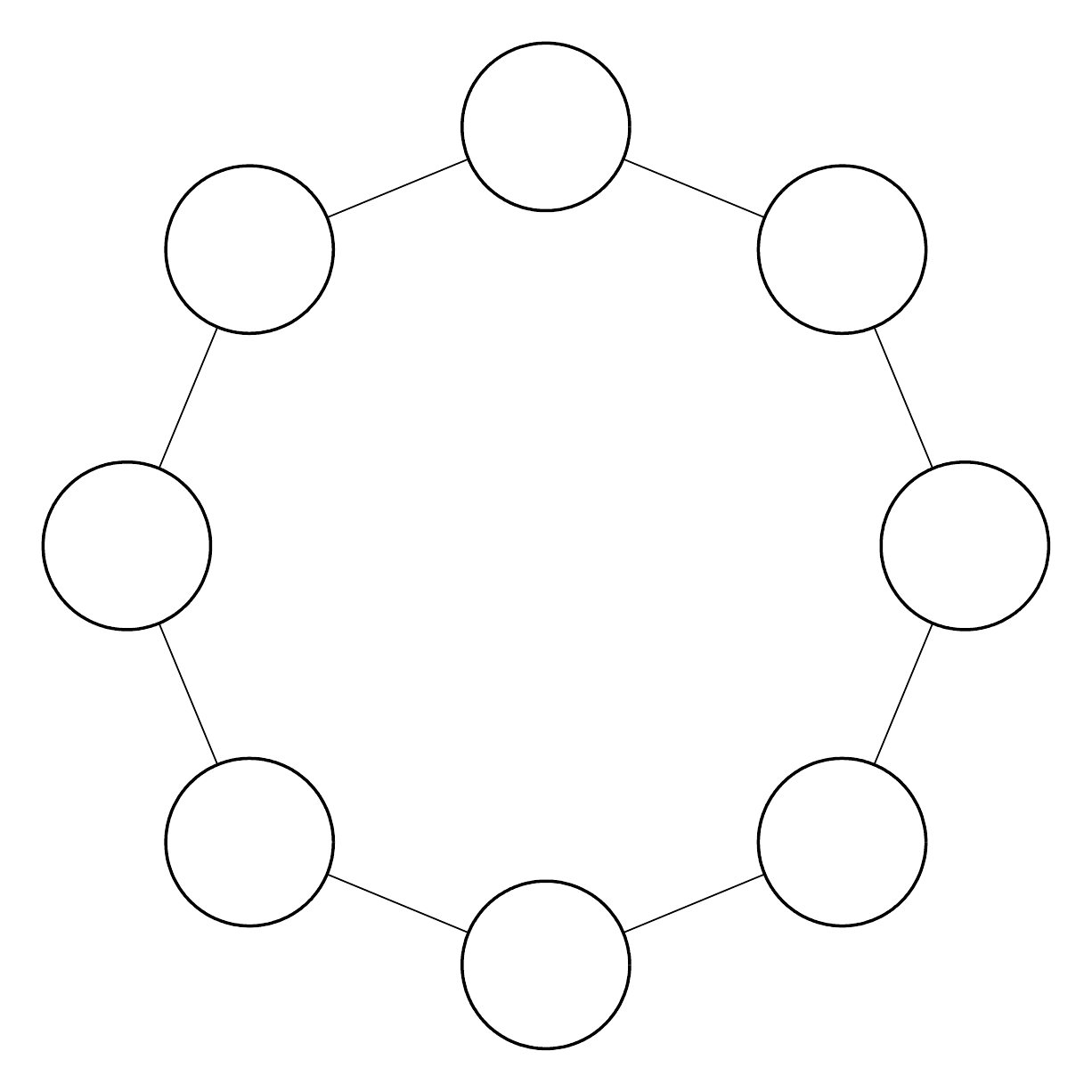}
   \begin{picture}(0,0)
     \put(-75.5,114){\makebox(10,10)[c]{1}}
     \put(-39,99){\makebox(10,10)[c]{2}}
     \put(-23.5,62){\makebox(10,10)[c]{3}}
     \put(-39,25){\makebox(10,10)[c]{4}}
     \put(-75.5,10){\makebox(10,10)[c]{5}}
     \put(-112,25){\makebox(10,10)[c]{6}}
     \put(-127.5,62){\makebox(10,10)[c]{7}}
     \put(-112,99){\makebox(10,10)[c]{8}}
     \put(-74.5,-10){\makebox(10,10)[c]{\emph{original order}}}
   \end{picture} \quad
   \includegraphics[scale=0.38]{reorder8}
   \begin{picture}(0,0)
     \put(-75.5,114){\makebox(10,10)[c]{1}}
     \put(-39,99){\makebox(10,10)[c]{3}}
     \put(-23.5,62){\makebox(10,10)[c]{5}}
     \put(-39,25){\makebox(10,10)[c]{7}}
     \put(-75.5,10){\makebox(10,10)[c]{8}}
     \put(-112,25){\makebox(10,10)[c]{6}}
     \put(-127.5,62){\makebox(10,10)[c]{4}}
     \put(-112,99){\makebox(10,10)[c]{2}}
     \put(-74.5,-10){\makebox(10,10)[c]{\emph{breadth-first order}}}
   \end{picture} \quad
\end{center}
\vspace*{7pt}
\caption{\label{fig:reorder8}
The vertex reordering scheme for $K=8$.}
\vspace*{-5pt}
\end{figure}

\begin{figure}
\begin{center}
   \includegraphics[scale=0.33]{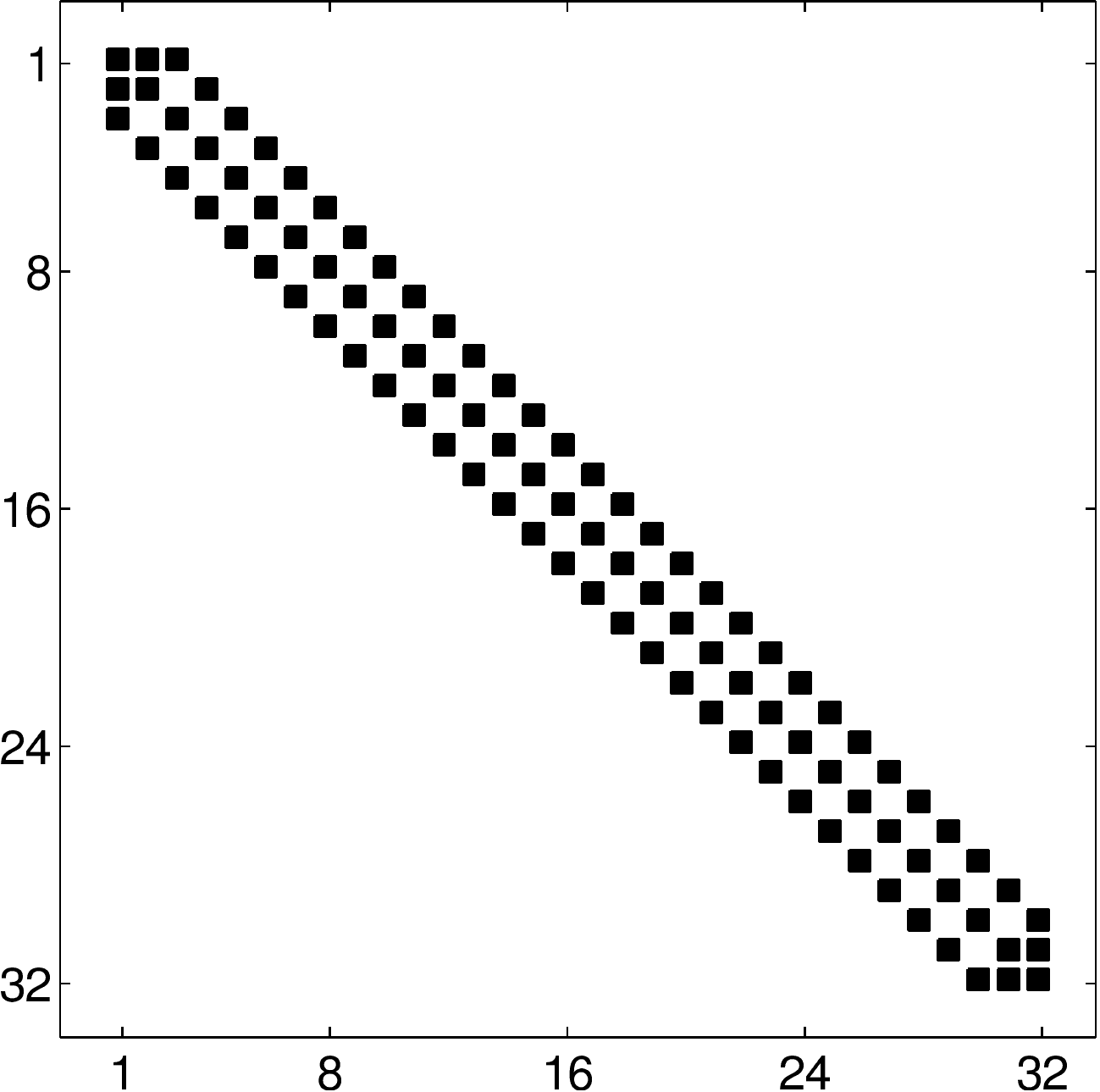} \qquad
   \includegraphics[scale=0.33]{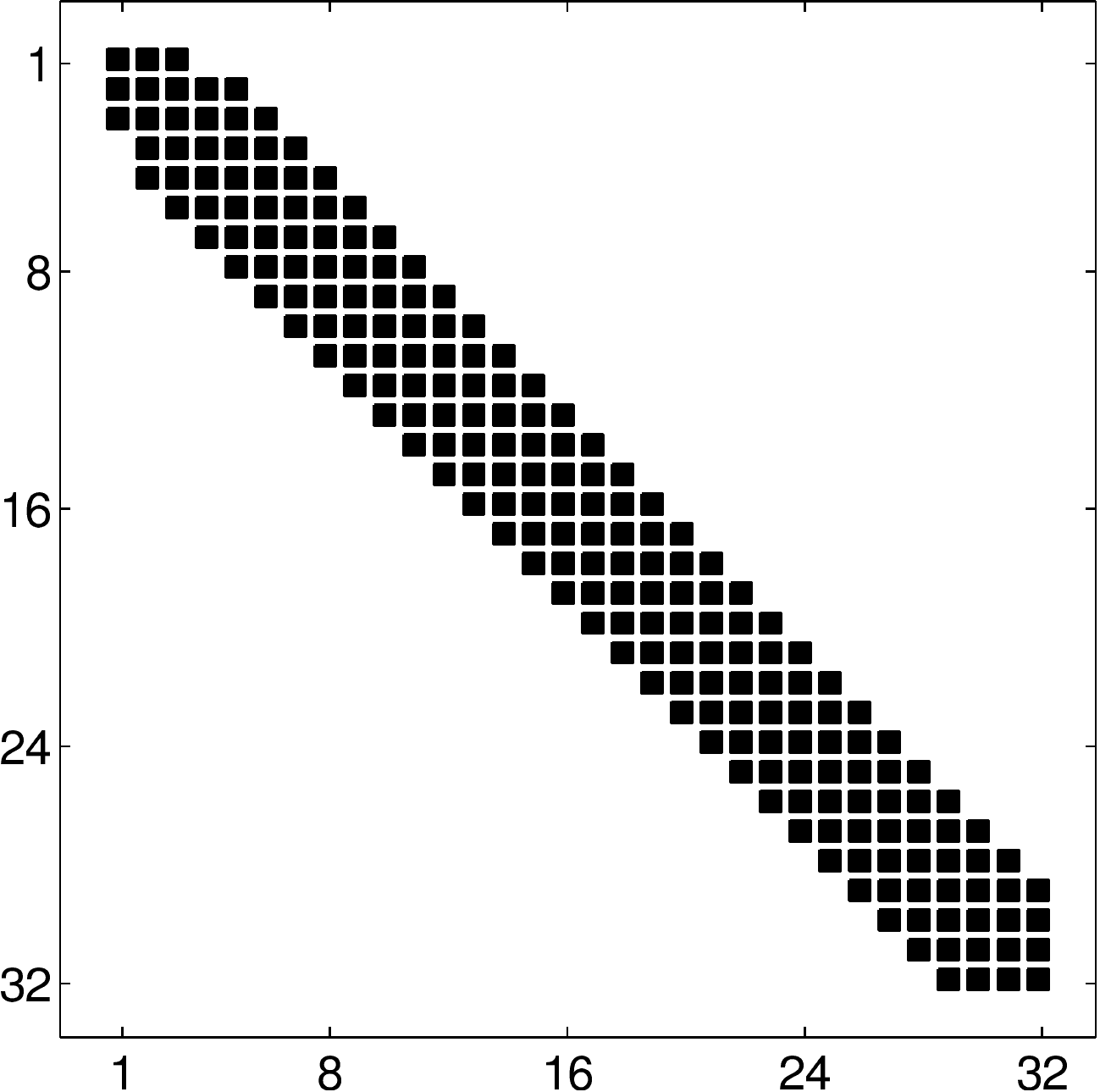}
\end{center}

\vspace*{-5pt}
\caption{\label{fig:spyplotsB}
Analogue of Figure~\ref{fig:spyplotsA} for the pentadiagonal matrix
from the breadth-first ordering (left) and the nonzeros that arise in
reduction of this matrix to tridiagonal form using the bulge-chasing approach (right).}
\vspace*{-5pt}
\end{figure}

One can compute a banded tridiagonalization
using the LAPACK software library's {\tt dsbtrd} routine,
or compute the eigenvalues directly with the banded eigensolver {\tt dsbev}~\cite{And99}.
(If $\BP\BJ_\pm\BP^*$ stored in sparse format, MATLAB's {\tt eig} command will identify 
the band form and find the eigenvalues in $O(K^2)$ time.)
We have benchmarked our computations in LAPACK with standard double precision
arithmetic and a variant compiled for quadruple precision.%
\footnote{\url{http://icl.cs.utk.edu/lapack-forum/viewtopic.php?f=2&t=2739}
gives compilation instructions.}
When $a_n=1$ and $b_n=0$ for all $n$, the eigenvalues of $\BJ_\pm$ 
are known in closed form~\cite{Gea69}.\ \ 
With the reordering scheme described above, LAPACK returns the
correct eigenvalues, up to the order of machine precision (roughly 
$10^{-16}$ for double precision and $10^{-34}$ for quadruple precision).
Figure~\ref{fig:time} compares the timing of the reordered scheme to
the traditional dense matrix approach.  In both double and
quadruple precision, the $O(K^2)$ performance of the reordered approach
offers a significant advantage.
(These timings were performed on a desktop with a 3.30~GHz Intel Xeon E31245 processor,
applied to the Fibonacci model described in the next section.
The numbers in the legend indicate the slope of a linear fit of the last five data points,
and each algorithm is averaged over the four Fibonacci parameters $\lambda = 1, 2, 3, 4$.)

\begin{figure}
\includegraphics[scale=0.6]{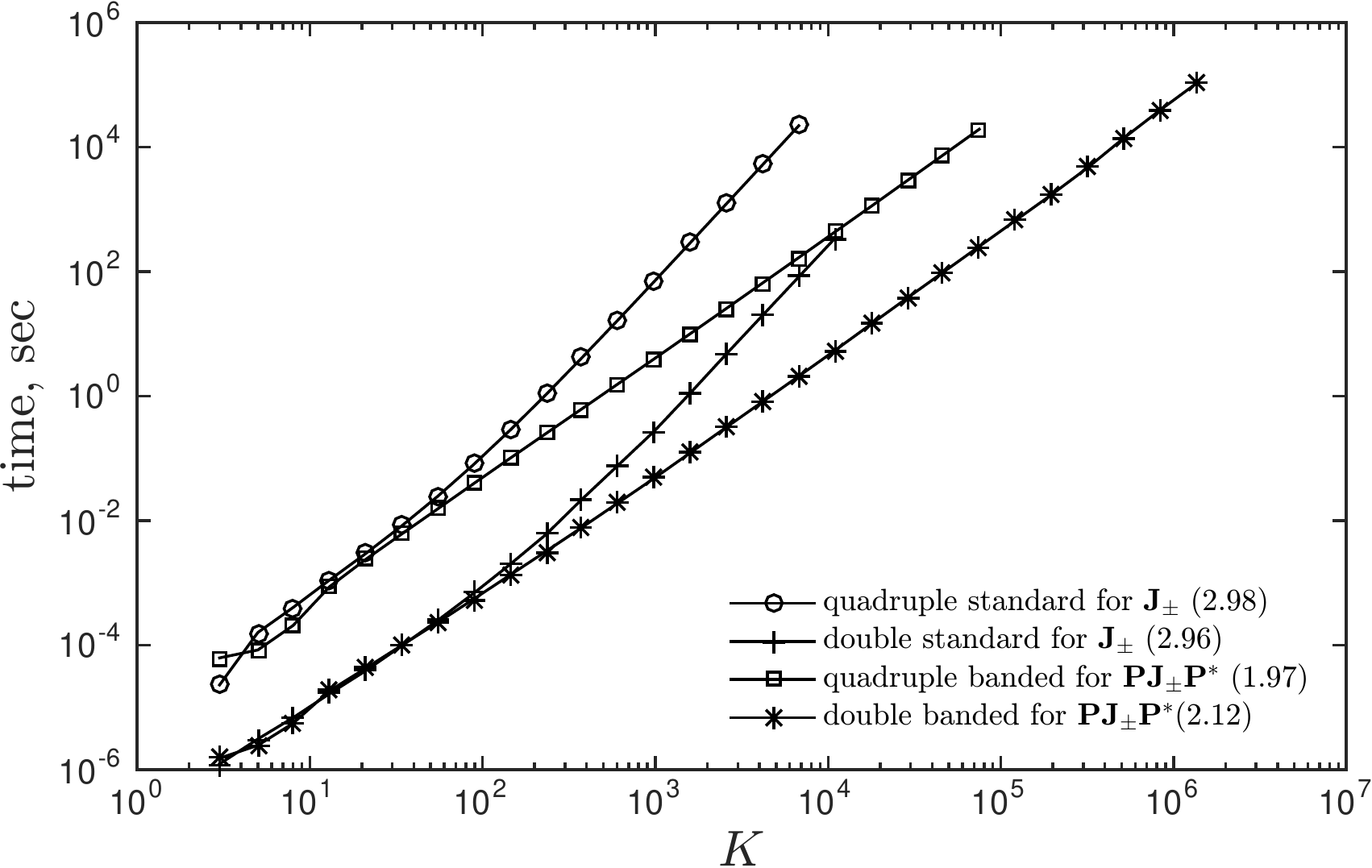}

\vspace*{-5pt}
\caption{\label{fig:time}
Performance of LAPACK's standard ({\tt dsyev}) and banded ({\tt dsbev}) 
symmetric eigensolvers in double and quadruple precision, 
applied to the Fibonacci model.}
\end{figure}
\section{Spectral theory for quasiperiodic Schr\"odinger operators} \label{sec:quasi} \label{ssec:subshfitspecth}

We focus on Jacobi operators that are \emph{discrete Schr\"odinger operators}, that is, those operators for which the off-diagonal terms satisfy $a_n =1$ for all $n$, and the \emph{potential} $\{b_n\}$ varies in a deterministic but non-periodic fashion.\footnote{In \emph{off-diagonal} models, $b_n = 0$ for all $n$, while the $a_n$ coefficients vary aperiodically, see, e.g., \cite{Mar12,Yes11}.} For several prominent examples the spectrum is a Cantor set; in other cases even a gross description of the spectral type has been elusive.

Periodic approximations lead to elegant covers of the spectra of the quasiperiodic operators,%
\footnote{Except for the Almost Mathieu operator, the potentials we describe are not quasiperiodic in the classical sense of being almost periodic sequences with finitely generated frequency modules.  However, the terminology is completely standard at this point.}
and while the arguments that produce these covers are now standard, there is not such a direct venue in the literature where this framework is quickly summarized.  Consequently, we recapitulate the essential arguments for the reader unfamiliar with this landscape; those seeking greater detail can consult the more extensive survey~\cite{DamSubshift}.\ \  In Section~\ref{sec:exp} we shall apply our algorithm to numerically compute the spectral covers described here.

A variety of quasiperiodic potentials have been investigated
in the mathematical physics literature; for a survey, see, e.g., \cite{DF}.  
Most prominent is the \emph{almost Mathieu operator}~\cite{Har55, Tho83}, with potential
\begin{equation} \label{eq:am}
b_n = 2 \lambda \cos(2\pi(n \alpha + \theta))
\end{equation}
for irrational $\alpha$, nonzero \emph{coupling constant} $\lambda$, and  \emph{phase} $\theta \in \R $.  The spectrum of the almost Mathieu operator is a Cantor set for all irrational  $\alpha$, every $\theta$, and every $\lambda \neq 0$~\cite{AvilaJito09}.\ \  Moreover, the Lebesgue measure of the spectrum is precisely $4\big|1-|\lambda|\big|$ whenever $\alpha$ is irrational \cite[Theorem~1.5]{AvilaKri06}.  Of particular interest is the Hausdorff dimension of the spectrum in the critical case $\lambda = 1$, about which very little is known; see \cite{Last94} for some partial results.

\emph{Sturmian potentials} take the form
\begin{equation} \label{eq:sturm}
 b_n = \lambda \chi_{[1-\alpha, 1)} ((n \alpha\ + \theta)  \, {\rm mod}\ 1 ),
\end{equation}
where $\chi_S$ is the indicator function on the set $S\subset \R$ 
and $\alpha$, $\lambda$, and $\theta$ play the same role as in~(\ref{eq:am}).  For such potentials the spectrum is a Cantor set of zero Lebesgue measure for all irrational $\alpha$, nonzero $\lambda $ and phases $\theta$~\cite{BIST89}.  It is conjectured that for a fixed nonzero coupling $\lambda$, the Hausdroff dimension of the spectrum is constant on a set of $\alpha$'s of full Lebesgue measure; currently, this is only known for $\lambda \geq 24$ \cite[Theorem~1.2]{DG14Hdim}.

\subsection{Quasiperiodic operators from substitution rules}

Alternatively, aperiodic potentials can be constructed from primitive substitution rules.  
Unlike the almost Mathieu and Sturmian cases, the spectral type of the operators is not well-understood.  Having spectrum of zero Lebesgue measure precludes the presence of absolutely continuous spectrum, a result known well before Damanik and Lenz proved zero-measure spectrum; see the main result of \cite{Kotani89RMP}. So far, numerous partial results exclude eigenvalues for particular substitution operators~\cite{BIST89,  Dam2000LMP, Dam2001AHP, DamLenz99CMP, HKS95}, and no results yet establish the existence of eigenvalues for any (two-sided) substitution operator.
The construction is slightly more involved than previous examples, but since seeking a better understanding of substitution potentials is a motivation of this work, we describe these objects in detail.

Let $\A \subset \R $ be a finite set, called the \emph{alphabet}.\footnote{We do not necessarily have to restrict our attention to alphabets consisting of real numbers, but this makes the definition of our subshift potentials in \eqref{def:subshiftpotential} somewhat simpler.} A \emph{substitution} on $\A$ is a rule that replaces elements of $\A$ by finite-length words over $\A$.\ \ For example, the \emph{period doubling} substitution, $S_{\mathrm{PD}}$, is defined by the rules $\a \mapsto \a\b$ and $\b \mapsto \a\a$ over the two-letter alphabet $ \A = \{\a,\b\}$~\cite{BBG91}.  
The \emph{Thue--Morse} substitution uses the rules $\a \mapsto \a\b$ and $\b \mapsto \b\a$~\cite{Bel90, DelPey91}.  The \emph{Fibonacci} potential~\cite{KKT83, OPRSS83} can be viewed as a Sturmian potential with $\alpha = (\sqrt{5}-1)/2$ and $\theta = 0$, or as a primitive substitution potential with rules $\a \mapsto \a\b$, $\b \mapsto \a$. (The equivalence of these definitions follows from \cite[Lemma 1b]{BIST89}.)  One can also study substitutive sequences on larger alphabets; for example, the \emph{Rudin--Shapiro} substitution is defined on the four-symbol alphabet $\A = \{\a,\b,\c,\d\}$ by the rules $\a \mapsto \a\b $, $ \b \mapsto \a\c $, $ \c \mapsto \d\b $, and $\d \mapsto \d\c$~\cite{Rud59, Shap51}.

Each of the aforementioned examples enjoys a property known as primitivity, which can be described informally as the existence of an iterate which maps each letter to a word containing the full alphabet.  More precisely, we say that a substitution $S$ is \emph{primitive} if there exists some $k \in \Z_+$ so that for every $\a,\b \in \A$, $\b$ is a subword of $S^k(\a)$. 

We now describe how a primitive substitution rule leads to a quasiperiodic Schr\"odinger operator,
focusing on period doubling as a concrete example.
To obtain an aperiodic sequence, start with the symbol $\a$ and form the sequence $w_k = S_{\mathrm{PD}}^k(\a)$ by iteratively applying the period doubling substitution rules, $\a \mapsto \a\b$ and $\b \mapsto \a\a$. The result is a sequence of finite words: $w_0 = \a$, $w_1 = \a\b$, $w_2 = \a\b\a\a$, $w_3 = \a\b\a\a\a\b\a\b$, and so on.  Notice that  $w_k$ is always a prefix of $w_{k+1}$.  In particular, there is a well-defined limiting sequence,
\begin{equation} \label{eq:subword}
x_{\mathrm{PD}} 
= \lim_{k \to \infty} S_{\rm PD}^k(\a)
= \a\b\a\a\a\b\a\b\a\b\a\a\a\b\a\a \ldots,
\end{equation}
with the property that $x_{\mathrm{PD}}$ is fixed by the period doubling substitution; sequences with this property are called \emph{substitution words} for the substitution.   

From a substitution word $x_{\rm PD}$ one can construct a quasiperiodic potential for a discrete Schr\"odinger operator.  Given a coupling constant $\lambda$, take
\begin{equation}\label{eq:subshiftpotdef}
\a = \lambda; \qquad \b = 0; \qquad b_n = \mbox{$n$th symbol of $x_{\rm PD}$}.
\end{equation}
One subtlety remains: the substitution word $x_{\mathrm{PD}}$ is a one-sided sequence, while our potentials are two-sided, i.e., we should specify $b_n$ for all $n\in\Z$.  To generate two-sided potentials, one considers an accumulation point of left-shifts of $x_{\mathrm{PD}}$; equivalently, one considers a two-sided sequence with the same local factor structure as $x_{\mathrm{PD}}$. The details of this construction follow.

 Suppose $S$ is a primitive substitution on the finite alphabet $\A \subset \R $ and $x \in \A^{\N} $ is some substitution word thereof, i.e., $S(x) = x$.  The \emph{subshift} generated by $S$ is the set of all two-sided sequences with the same local factor structure as $x$.  More precisely,
\begin{equation}\label{eq:subshiftdef}
\Omega_S
= \left\{ \omega \in \A^{\Z} : \omega_n \cdots \omega_m \text{ is a subword of } x \text{ for every } n \leq m \right\}.
\end{equation}
Using primitivity of $S$, one can check that this definition of $\Omega_S$ does not depend upon the choice of substitution word.  One can also generate the set $\Omega_S$ via the following dynamical procedure.  First, endow the sequence space $\A^{\Z}$ with some metric that induces the product topology thereupon, e.g.,
\begin{equation}\label{def:subshift:metric}
d(\omega,\omega')
=
\sum_{n \in \Z} \frac{1-\delta_{\omega_n,\omega_n'}}{2^{|n|+1}},
\end{equation}
where $\delta_{a,b}$ denotes the usual Kronecker delta symbol.  The sequences $\omega$ and $\omega'$ are close with respect to the metric in \eqref{def:subshift:metric} if and only if they agree on a large window centered at the origin, so this metric does indeed induce the topology of pointwise convergence on $\Omega_S$.  One can check that $\Omega_S$ defined by \eqref{eq:subshiftdef} coincides precisely with the set of limits of convergent subsequences of the sequence $ (T^nx)_{n=1}^\infty $, where $T$ denotes the usual left shift $ (T\omega)_n = \omega_{n+1}$. (There is a minor technicality here: since $x$ is a one-sided sequence $(T^n x)_k$ is only well-defined for $n > |k|$ when $k$ is negative.)

Primitivity of $S$ implies that the topological dynamical system $(\Omega_S,T)$ is strictly ergodic \cite[Proposition~5.2 and Theorem~5.6]{queff}. For each $\omega \in \Omega_S$, one obtains a Schr\"odinger operator $H_\omega$ on $\ell^2(\Z)$ defined via
\begin{equation} \label{def:subshiftpotential}
(H_\omega \psi)_n = \psi_{n-1} + \omega_n \psi_n + \psi_{n+1}
\end{equation}
for each $n \in \Z$.\ \ (The right hand side of \eqref{def:subshiftpotential} makes sense, since we chose $\A \subseteq \R$.) Minimality of $(\Omega_S,T)$ implies $\omega$-invariance of the spectrum.

\begin{prop} Given $\Omega_S$ and $H_\omega$ as above, there is a uniform compact set $\Sigma$ with the property that $\sigma(H_\omega) = \Sigma$ for every $\omega \in \Omega_S$.
\end{prop}  

\begin{proof}
Given $\omega,\omega' \in \Omega_S$, there is a sequence $\{n_k\} \subset \Z$ with the property that $T^{n_k} \omega \to \omega'$ as $k \to \infty$. (This is a consequence of minimality of $(\Omega_S,T)$, which follows from \cite[Proposition~4.7]{queff}, for example.)  In particular,
$$
H_{\omega'} = \slim_{k\to\infty} H_{T^{n_k}\omega}.
$$
By a standard strong approximation argument (e.g., \cite[Theorem~VIII.24]{RS80}), 
\begin{equation}\label{strong:app}
\sigma(H_{\omega'}) 
\subseteq
\bigcap_{\ell=1}^\infty \overline{\bigcup_{k=\ell}^\infty \sigma(H_{T^{n_k}\omega})}
=
\sigma(H_\omega).
\end{equation}
By symmetry, we can run the previous argument with the roles of $\omega$ and $\omega'$ reversed, so $\sigma(H_\omega) = \sigma(H_{\omega'})$.
\end{proof}

This is the upshot of the previous proposition: if we want to study the spectrum of a two-sided substitutive Schr\"odinger operator \emph{as a set}, then we can work with any member of the associated subshift.

One can avoid the construction needed to generate two-sided potentials by considering instead 
\emph{half-line} Jacobi operators, $\CJ_+:\ell^2(\N) \to \ell^2(\N)$.  These operators are defined entrywise by \eqref{eq0} for $n \in \N$ with the normalization $\psi_0 = 0$ to make $(\CJ_+\psi)_1$ well defined.  If $\CJ_+$ denotes the Jacobi operator defined by $a_n \equiv 1$ and \eqref{eq:subshiftpotdef}, and $\CJ$ denotes a two-sided Jacobi operator generated in one of the two fashions just described, then by \cite[Theorem~7.2.1]{Sim11},
\begin{equation} \label{eq:shiftlimitspec}
\sigma_{\ess}(\CJ_+)
=
\sigma(\CJ),
\end{equation}
where $\sigma_{\ess}(\cdot)$ denotes the essential spectrum.  In particular, $\sigma(\CJ_+)$ and $\sigma(\CJ)$ have the same Hausdorff dimension, since \eqref{eq:shiftlimitspec} implies that they coincide up to a countable set of isolated points.

\subsection{Periodic approximations}

All the classes of quasiperiodic models we have described have natural periodic approximations.  For the almost Mathieu and Sturmian cases: replace the irrational $\alpha$ with a rational approximant; for the substitution rules: pick a starting symbol, generate a string from finitely many applications of the substitution rules, and repeat that string periodically.
We seek high fidelity numerical approximations to these periodic spectra, as a vehicle for understanding properties of quasiperiodic models, such as the fractal dimension of the spectrum.

In the case of substitution rules, we shall describe how periodic approximations lead to an \emph{upper bound} on the spectrum of the quasiperiodic operator.  Again we focus on the period doubling substitution.  Fix $S_{\mathrm{PD}}$, $x_{\mathrm{PD}}$ as before, choose some $\omega \in \Omega_{\mathrm{PD}}$, and put $w_k^\a = S^k(\a)$ and $w_k^\b = S^k(\b)$. We generate periodic Schr\"odinger operators $H_k^\a$ and $H_k^\b$ by repeating the strings $w_k^\a$ and $w_k^\b$ periodically, giving $H_k^\a$ and $H_k^\b$ with coefficients having period $2^k$.  Moreover, we choose $H_k^\a$ and $H_k^\b$ in such a way that
$$
H_\omega
=
\slim_{k \to \infty} H_k^\a
=
\slim_{k \to \infty} H_k^\b.
$$
Define $\Sigma_k^\a \equiv \sigma(H_k^\a)$ and $\Sigma_k^\b \equiv \sigma(H_k^\b)$  for all $k \geq 0 $.  Application of strong approximation, as before, implies
\begin{equation} \label{eq:perstrongapprox}
\Sigma
\subseteq
\bigcap_{n=1}^\infty \overline{\bigcup_{k=n}^\infty \Sigma_k^\a},
\quad
\Sigma
\subseteq
\bigcap_{n=1}^\infty \overline{\bigcup_{k=n}^\infty \Sigma_k^\b}.
\end{equation}
These unions are too large to be computationally tractable, but the hierarchical structure of the periodic approximations saves the day.  In these cases the monodromy matrices~(\ref{eq:monod}) take the forms
\begin{eqnarray*}
M_k^\a(E)
&=& \begin{bmatrix} E - w_k^\a(2^k) & -1 \\ 1 & 0\end{bmatrix}
\cdots
\begin{bmatrix} E - w_k^\a(1) & -1 \\ 1 & 0 \end{bmatrix} \\[.5em]
M_k^\b(E)
&=& \begin{bmatrix} E - w_k^\b(2^k) & -1 \\ 1 & 0 \end{bmatrix}
\cdots
    \begin{bmatrix} E - w_k^\b(1) & -1 \\ 1 & 0\end{bmatrix},
\end{eqnarray*}
where now the subscript $k$ denotes the $k$th iteration of the substitution rule, and hence
a full period of length $K = 2^k$.
Recalling the discriminant condition~(\ref{eq:disc}), define
$$
x_k(E) \equiv \tr(M_k^\a(E)),
\quad
y_k(E) \equiv \tr(M_k^\b(E)).
$$
As in~(\ref{eq:disc}), these functions encode the spectra of the periodic approximants,
in that $ \Sigma_k^\a = \{ E : |x_k(E)| \leq 2 \} $ and $ \Sigma_k^\b = \{ E : |y_k(E)| \leq 2 \} $.  
The  rules of the period doubling substitution imply
\begin{align}
\label{pd:matrec1}
M_{k+1}^\a & = M_k^\b M_k^\a \\
\label{pd:matrec2}
M_{k+1}^\b & = M_k^\a M_k^\a.
\end{align}
Applying the Cayley--Hamilton theorem ($M_{k}^2 - \tr(M_k)M_k+I=0$) to \eqref{pd:matrec1} and \eqref{pd:matrec2} yields
\begin{align}
\label{pd:trace1}
x_{k+1} & = x_k y_k -2 \\
\label{pd:trace2}
y_{k+1} & = x_k^2 -2,
\end{align}
given in~\cite[eq.~(1.9)]{BBG91}.  Now \eqref{pd:trace1} and \eqref{pd:trace2} imply $\Sigma_{k+1}^\a \cup \Sigma_{k+1}^\b \subseteq \Sigma_k^\a \cup \Sigma_k^\b$ for all $k \geq 0$, thus reducing~\eqref{eq:perstrongapprox} to the more tractable
\begin{equation} \label{pd:covers}
\llap{\mbox{\small [Period Doubling]}\hspace*{5em}}
\Sigma
\subseteq
\bigcap_{k=1}^\infty \Sigma_k^\a \cup \Sigma_k^\b.
\end{equation}
The spectrum for the Thue--Morse substitution can be expressed in the same way, except that substitution rule replaces~(\ref{pd:trace1})--(\ref{pd:trace2}) with 
\begin{align}
\label{tm:trace1}
x_{k+1} &= x_{k-1}^2(x_k - 2) + 2 \\
\label{tm:trace2}
y_{k+1} & = x_{k+1}.
\end{align}
Reasoning as in the period doubling case, \eqref{tm:trace1}--\eqref{tm:trace2} implies
\begin{equation}\label{tm:covers}
\llap{\mbox{\small [Thue--Morse]}\hspace*{6em}}
\Sigma
\subseteq
\bigcap_{k=1}^\infty \Sigma_k^\a \cup \Sigma_{k+1}^\a.
\end{equation}
Notice that \eqref{pd:covers} and \eqref{tm:covers} are not identical.  In general, the covers of the spectrum obtained via this approach depend quite strongly on the chosen substitution.  Note, however, that the same formula \eqref{tm:covers} holds for the Fibonacci substitution. The recursive relationships in \eqref{pd:trace1}--\eqref{pd:trace2} and \eqref{tm:trace1}--\eqref{tm:trace2} are known as the \emph{trace maps} for the period doubling and Thue--Morse potentials. Similar trace maps can be constructed for arbitrary substitutions; see \cite[Theorem~1]{AvBer93} and \cite[Theorem~1]{ABG94} for their description.
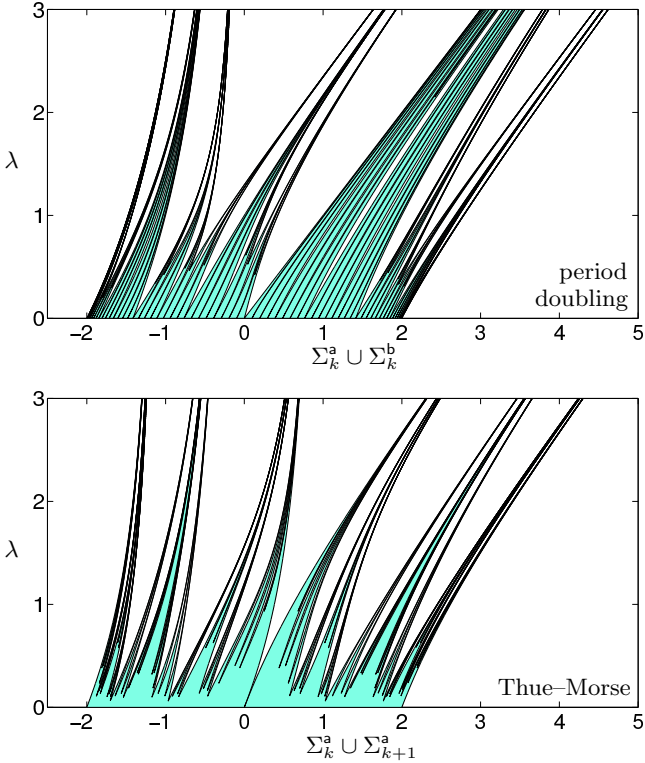
\begin{figure}[t!]
\begin{center}
\includegraphics[scale=0.45]{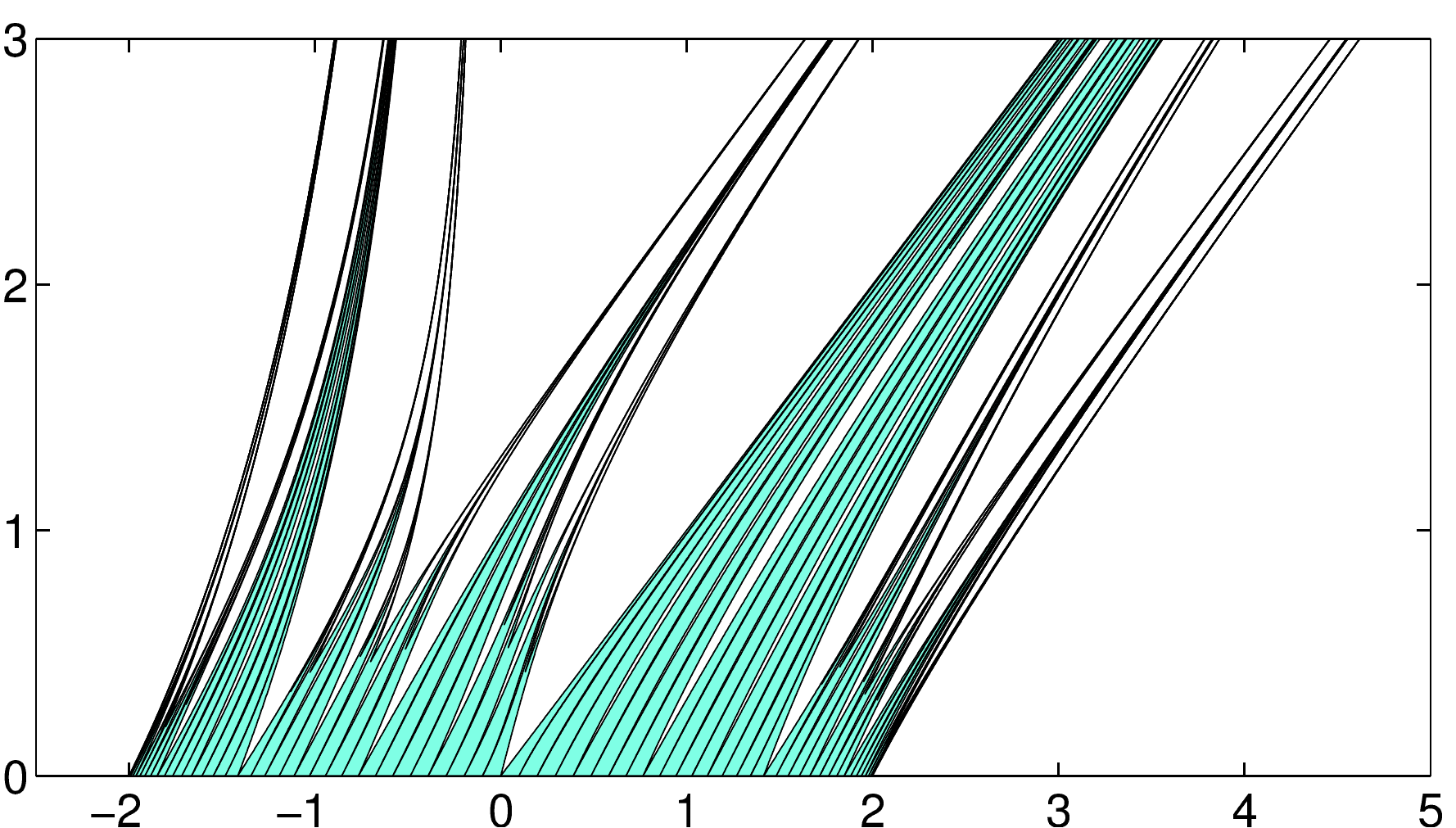}
\begin{picture}(0,0)
\put(-37,23){\small period}
\put(-46,13){\small doubling}
\put(-130,-9){\small $\Sigma_k^\a \cup \Sigma_k^\b$}
\put(-245,65){\small $\lambda$}
\end{picture}
\\[1.5em]
\includegraphics[scale=0.45]{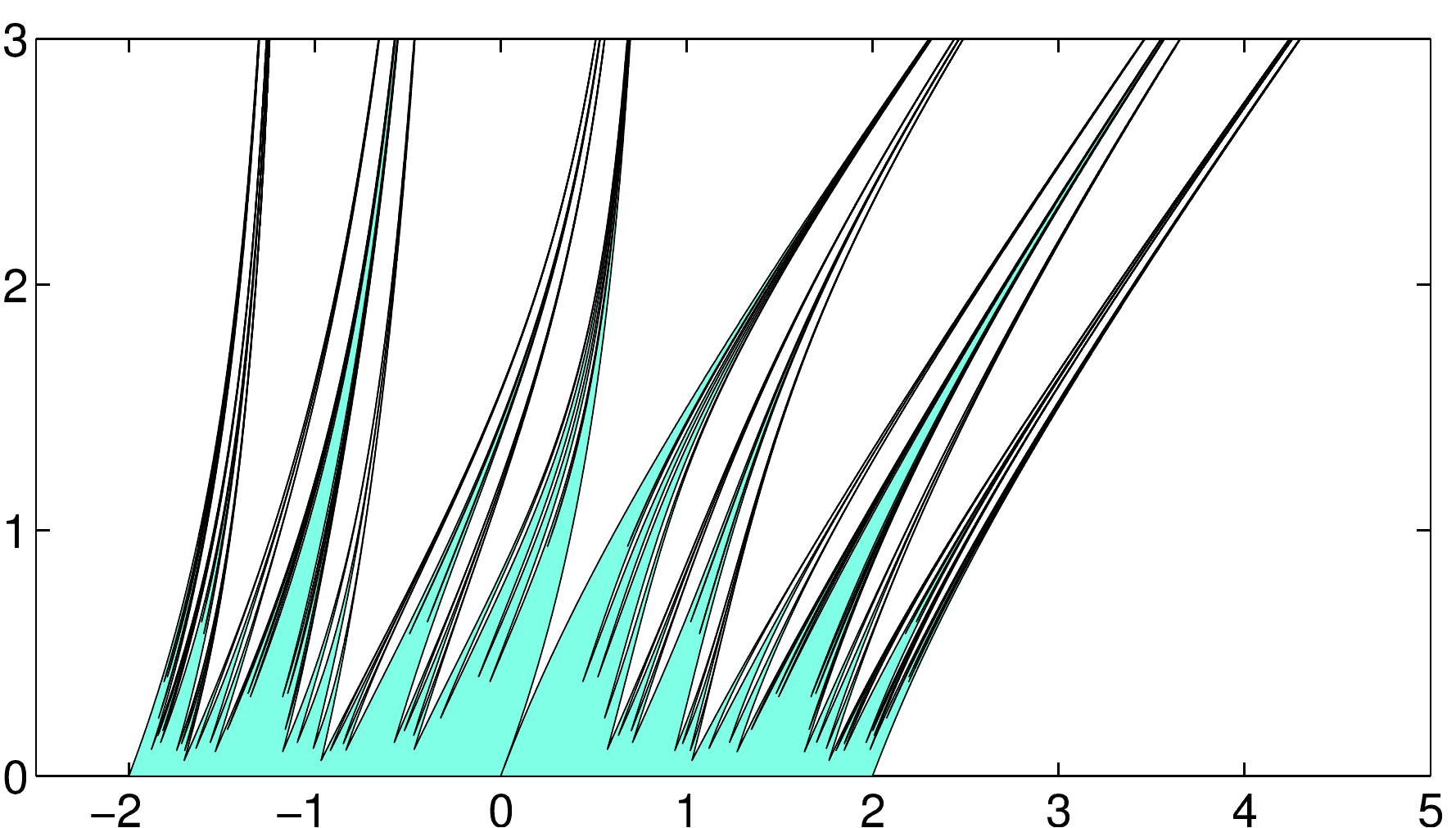}
\begin{picture}(0,0)
\put(-60,13){\small Thue--Morse}
\put(-132,-9){\small $\Sigma_k^\a \cup \Sigma_{k+1}^\a$}
\put(-245,65){\small $\lambda$}
\end{picture}
\end{center}

\caption{\label{fig:follicle}
Spectral covers for the period doubling ($\Sigma_k^\a \cup \Sigma_k^\b$) 
and Thue--Morse ($\Sigma_k^\a \cup \Sigma_{k+1}^\a$) potentials as a 
function of the coupling constant $\lambda$, for $k=7$.
}
\end{figure}

Figure~\ref{fig:follicle} depicts the covers $\Sigma_k^\a \cup \Sigma_k^\b$ 
for the period doubling potential and $\Sigma_k^\a \cup \Sigma_{k+1}^\a$ for
Thue--Morse as $\lambda$ increases from zero with $k=7$, 
giving an impression of how rapidly the covering intervals shrink 
as $\lambda$ increases.  Notice the quite different nature of the 
spectra obtained from these two substitution rules.
Similar figures for the Fibonacci spectrum are shown in~\cite{DEG12}.
(The algorithm described in the last section enables such computations for much larger
values of $k$, but the resulting figures become increasingly difficult to
render due to the number and narrow width of the covering intervals.)

\section{Numerical calculation of periodic covers} \label{sec:exp}

In this section, we apply the algorithm described in Section~\ref{sec:alg} to 
compute the spectral covers for quasiperiodic operators from substitution rules
described in Section~\ref{sec:quasi}.  We begin with some calculations that emphasize the need for higher precision arithmetic to resolve the spectrum, then estimate various spectral quantities for these quasiperiodic operators.

\subsection{Necessity for extended precision}

High fidelity spectral approximations for quasiperiodic Schr\"odinger operators
are tricky to compute due not only to the complexity of the eigenvalue problem,
but also considerations of numerical precision.  The LAPACK symmetric eigenvalue
algorithms are expected to compute eigenvalues of the $K\times K$ matrix $\BJ_\pm$ 
accurate to within $p(K) \|\BJ_\pm\| \eps_{\rm mach}$~\cite[p.~104]{And99}, 
where $p(K)$ is a ``modestly growing function of $K$'' and $\eps_{\rm mach}$ 
denotes the machine epsilon value for the floating point arithmetic system
(on the order of $10^{-16}$ for double precision and $10^{-34}$ for quadruple precision~\cite{IEEE}).
Figure~\ref{fig:intlencompare} shows the width of the smallest interval in the
approximation $\Sigma_k^\a$ for the period doubling and Thue--Morse potentials,
as computed in double and quadruple precision arithmetic for various values of~$\lambda$. 
As $\lambda$ increases, this smallest interval shrinks ever quicker.
Comparing double and quadruple precision values, one sees that 
for even moderate values of $K=2^k$, double precision is unable to 
accurately resolve the spectrum.%
\footnote{The calculations of fractal dimension and gaps that follow were performed
in quadruple precision arithmetic, and we generally restricted the values of $\lambda$ and $k$ 
to obtain numerically reliable results.  In the event the numerical results 
violate the ordering of eigenvalues of $\BJ_+$ and $\BJ_-$ in~(\ref{eq:interleave}),
any offending interval is replaced by one of width $20\kern1pt\eps_{\rm mach}$ centered at 
the midpoint of the computed eigenvalues.}
(Similar numerical errors will be observed for the original
matrix $\BJ_\pm$.  Since it only permutes matrix entries,
our algorithm does not introduce any new instabilities.)

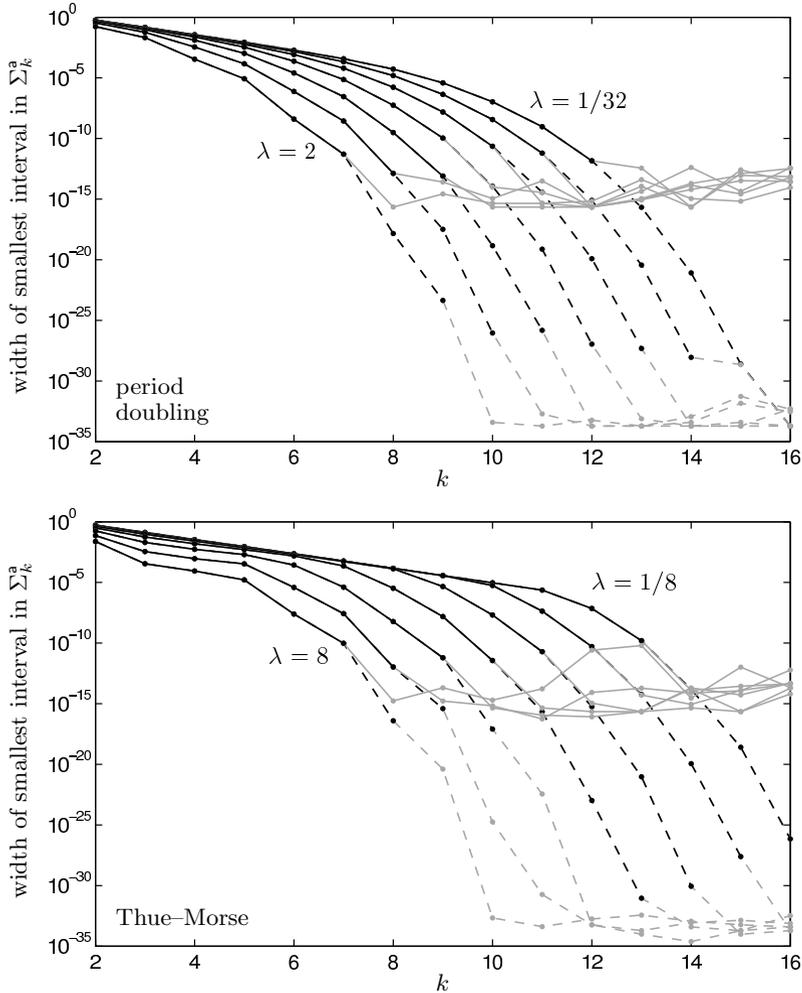
\begin{figure}
  \includegraphics[scale=0.60]{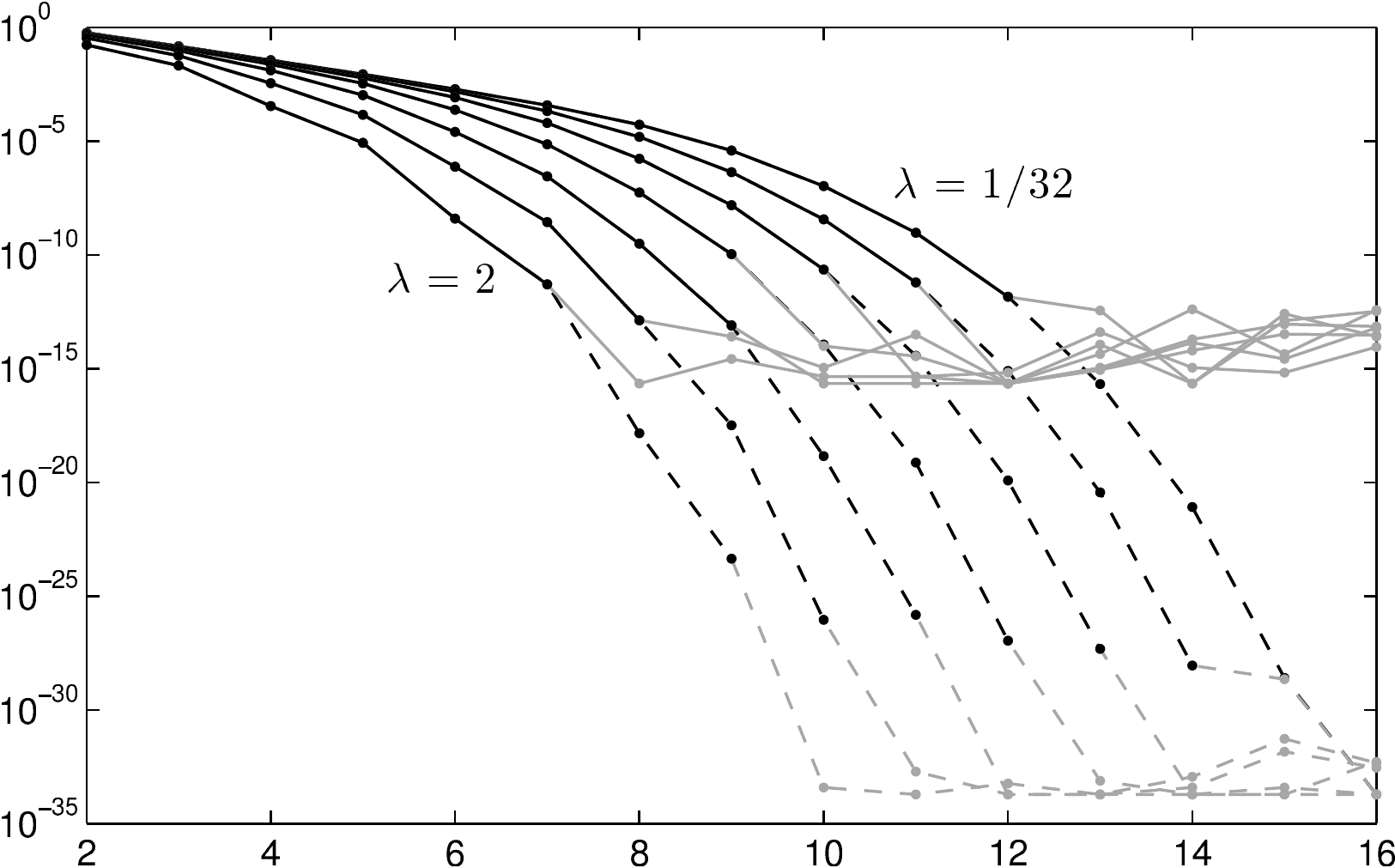}
\begin{picture}(0,0)
\put(-140,-8){\small $k$}
\put(-260,27){\small period}
\put(-260,17){\small doubling}
\put(-300,27){\rotatebox{90}{\small width of smallest interval in $\Sigma_k^\a$}}
\end{picture}

\vspace*{1.5em}
  \includegraphics[scale=0.60]{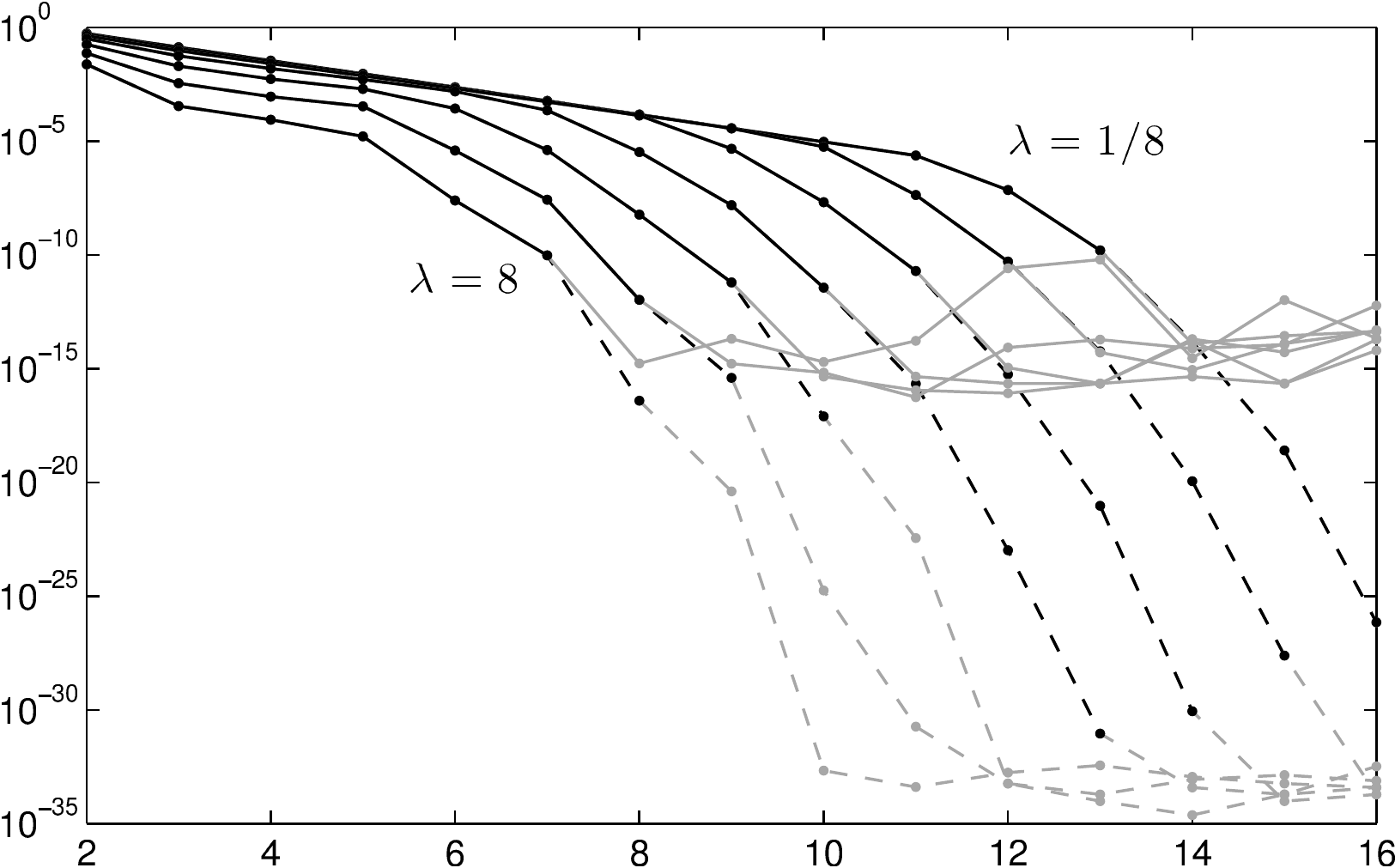}
\begin{picture}(0,0)
\put(-140,-8){\small $k$}
\put(-260,17){\small Thue--Morse}
\put(-300,27){\rotatebox{90}{\small width of smallest interval in $\Sigma_k^\a$}}
\end{picture}

\vspace*{2pt}
  \caption{\label{fig:intlencompare}
  Minimum interval length in $\Sigma_k^\a$ for a range of~$\lambda$ values (powers of two) for the period doubling and Thue--Morse potentials.
The solid line shows double precision, the dashed line quadruple precision.  
Black data points are believed to be correct to plotting accuracy; gray ones are not.}
\end{figure}

\subsection{Approximating fractal dimensions}

The quasiperiodic spectrum $\Sigma$ is known to be a Cantor set for 
the Fibonacci, period doubling, and Thue--Morse potentials for
all $\lambda>0$, 
suggesting calculation of the fractal dimensions of these sets
as functions of $\lambda$.
We begin with two standard definitions; see, e.g., \cite{Fal03,PT93a}.

\begin{defn}
  Given $A\subset \R$ and some $\alpha\in[0,1]$, let
  \[
  h^\alpha(A) \equiv \lim_{\Delta \rightarrow 0} \inf_{\text{$\Delta$-covers}} \sum_{m \geq 1} |B_m|^\alpha,
  \]
 where a $\Delta$-cover of $A$ is defined to be collection $\{B_m\}_{m\ge 1}$ of intervals such that $A\subset \bigcup_{m\ge 1} B_m$ and $|B_m|<\Delta$ for each $m$.
The \emph{Hausdorff dimension} of $A$ is then
  \[
  \dim_H (A) \equiv \inf \{\alpha : h^\alpha(A) < \infty \}.
  \]
\end{defn}

\begin{defn}
   Given $A\subset \R$, let $N_A(\eps)$ denote the number of intervals of the form
   $[j\eps, (j+1)\eps)$, $j\in\Z$, that have a nontrivial intersection with $A$.
   The upper and lower box counting dimensions of $A$ are 
\[ \bcd^\pm(A) \equiv {\lim \kern-2pt {\begin{array}{c} \sup \\[-2pt] \inf\end{array}}}_{\hspace*{-30pt}\raisebox{-4pt}{$\scriptstyle{\eps \to 0}$}} \hspace*{20pt} {\log N_A(\eps) \over \log 1/\eps};
\]
when these agree, the result is the \emph{box-counting dimension} of $A$, $\bcd(A)$.
\end{defn}

The analysis in the last section provides natural covering sets for 
$\Sigma$, which we denote as $C_k$  (i.e., for period doubling,
$C_{k} \equiv \Sigma_k^\a \cup \Sigma_k^\b$; for Fibonacci and Thue--Morse, 
$C_{k} \equiv \Sigma_k^\a \cup \Sigma_{k+1}^\a$).
We estimate $\dim_H(\Sigma)$ using a heuristic proposed by Halsey et al.~\cite{HJKPS86}.
Enumerate the $n_k$ intervals in $C_k$ as $\{B_{k, m}\}$:
\[
C_{k} = \bigcup_{m=1}^{n_k} B_{k, m}.
\]
\begin{algorithm}[h!]
\title{\textbf{Approximating the Hausdorff dimension}}
\label{alg:HD}
\begin{algorithmic}[1]
\State {\bf Input}: two consecutive covers $C_{k, \lambda}$, $C_{k + 1, \lambda}$ of the Cantor spectrum.
\State Construct the function $f(\alpha) \equiv  \sum_{m = 1}^{n_k} |B_{k,m}|^\alpha - \sum_{m = 1}^{n_{k+1}} |B_{k+1,m}|^\alpha$.
\State Compute root $\alpha_k$ of $f(\alpha)$ in the interval $[0,1]$.
\State {\bf Output}: $\alpha_k$ as the approximation to the dimension.
\end{algorithmic}
\end{algorithm}

For the Fibonacci case,
Damanik et al.~\cite{DEGT08} proved upper and lower bounds on $\dim_H(\Sigma)$ for $\lambda\ge 8$ in terms of the functions $S_u(\lambda) = 2\lambda + 22$ and $S_l(\lambda) = \frac{1}{2}\big( \lambda-4 + \sqrt{(\lambda-4)^2 - 12} \big)$.  
To benchmark our method,  we compare the approximate dimension with the upper and lower bounds, 
as shown in Figure~\ref{fig:dimHfib}, indeed obtaining satisfactory results. 
Figure~\ref{fig:dimHpd} shows estimates to $\dim_H(\Sigma)$ drawn from spectral
covers for $k=6,\ldots, 10$, 
with good convergence in $k$ in the small $\lambda$ regime.
The narrow covering intervals for large $\lambda$  and $k$ 
pose a significant numerical challenge, even in quadruple precision arithmetic.

\begin{figure}[t!]
  \includegraphics[scale=0.45]{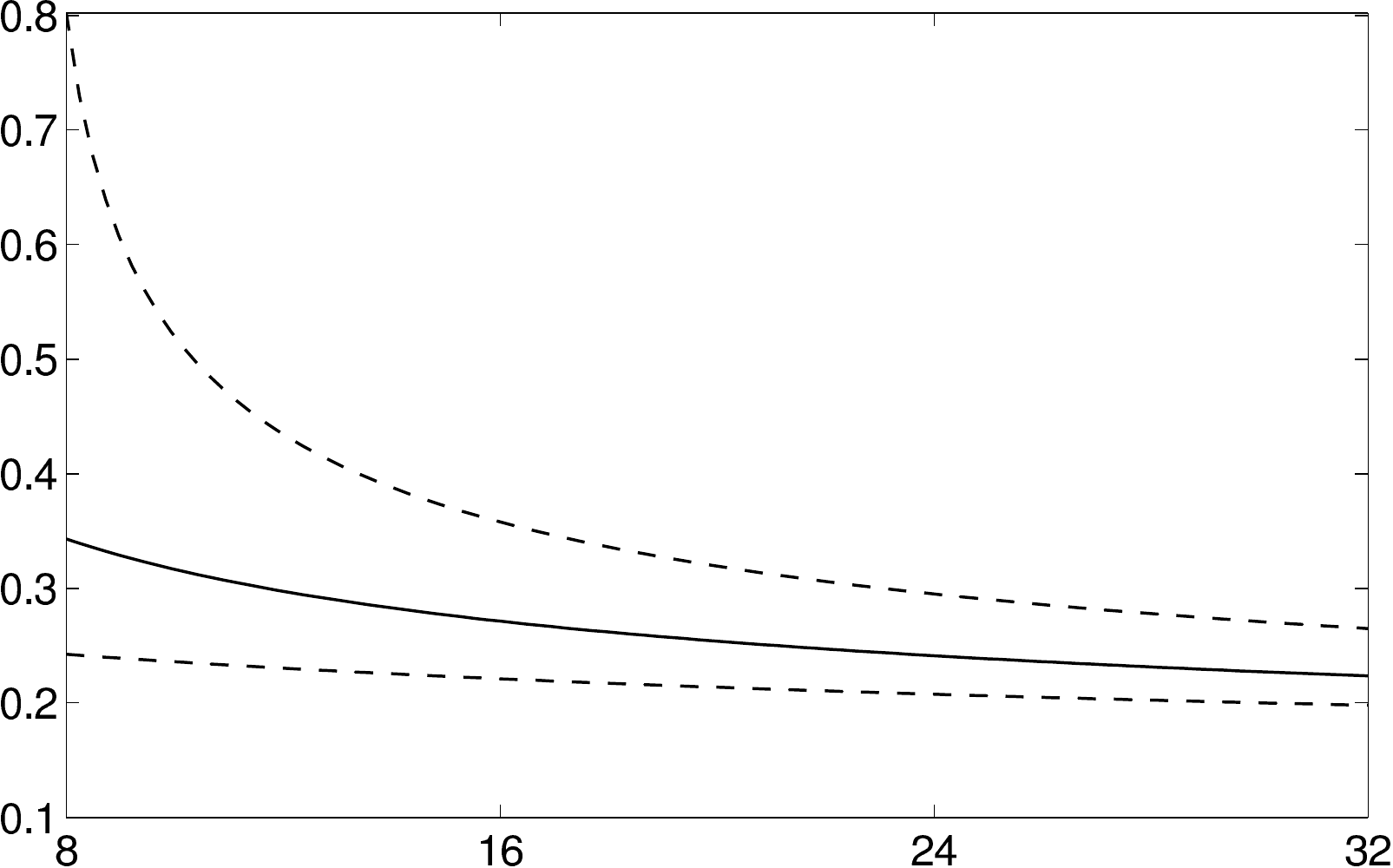}
\begin{picture}(0,0)
\put(-105,-6){\small $\lambda$}
\put(-153,62){\small ${\log(1+\sqrt{2}) \over \log S_l(\lambda)}$}
\put(-190,17){\small ${\log(1+\sqrt{2}) \over \log S_u(\lambda)}$}
\put(-185,48){\small \rotatebox{-10}{$\dim_H(\Sigma)$}}
\end{picture}
  \caption{ \label{fig:dimHfib}
Estimates of $\dim_H(\Sigma)$ for the Fibonacci operator (solid line) computed from $\Sigma_{k}^\a \cup \Sigma_{k+1}^\a$ ($k = 15,16$), obeying the upper and lower bounds from~\cite{DEGT08}.}
\end{figure}
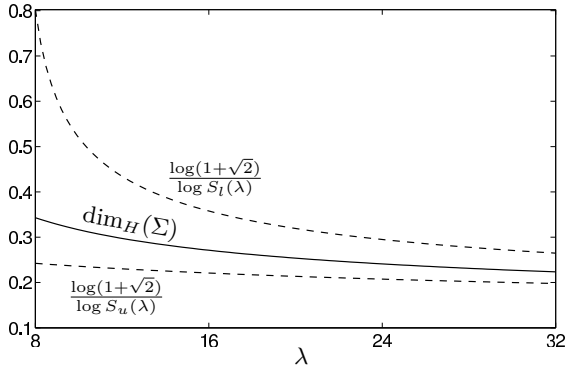

\begin{figure}[t!]
  \includegraphics[scale=0.45]{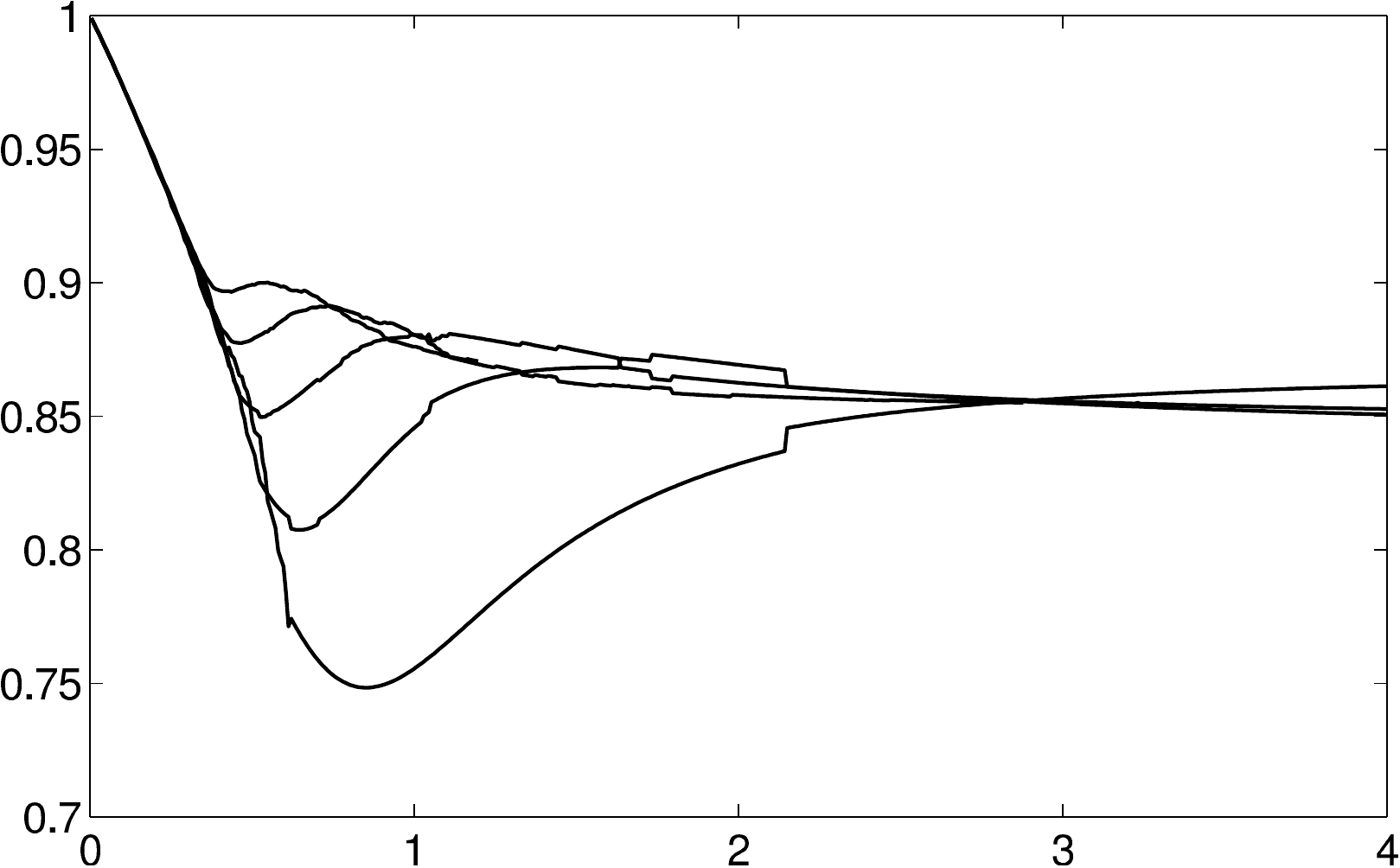}
\begin{picture}(0,0)
\put(-105,-8){\small $\lambda$}
\put(-150,26){\small \rotatebox{40}{$\scriptstyle{k=6}$}}
\put(-160,52){\small \rotatebox{45}{$\scriptstyle{k=7}$}}
\put(-170,65){\small \rotatebox{40}{$\scriptstyle{k=8}$}}
\put(-178,73.5){\small \rotatebox{30}{$\scriptstyle{k=9}$}}
\put(-182,90){\small \rotatebox{0}{$\scriptstyle{k=10}$}}
\end{picture}
  \caption{\label{fig:dimHpd} 
Estimates of $\dim_H(\Sigma)$ for the period-doubling operator using consecutive 
covers $\Sigma_{k}^\a \cup \Sigma_{k}^\b$ and $\Sigma_{k+1}^\a \cup \Sigma_{k+1}^\b$, for
$k=6,\ldots, 10$.  (The $k=9$ and $k=10$ plots do not show the full range of $\lambda$
values, due to the numerical challenge of working with large $k$ and $\lambda$.)}
\end{figure}
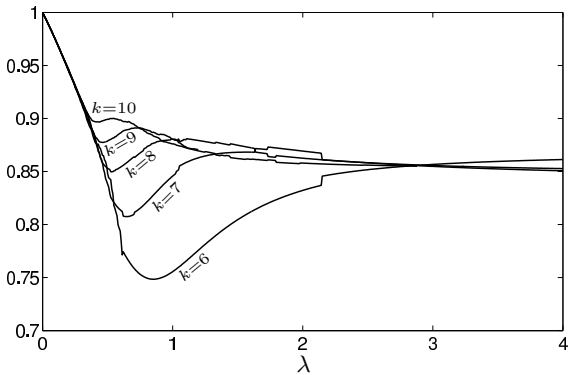

The bounds from~\cite{DEGT08} imply $\dim_H(\Sigma) \to 0$ 
like $\log(1+\sqrt{2})/\log(\lambda)$ as $\lambda\to\infty$ for the Fibonacci case.
In contrast, Liu and Qu~\cite{LiuQu} recently showed that for the Thue--Morse operator,
$\dim_H(\Sigma)$ is bounded away from zero as $\lambda\to\infty$.
Interestingly, the approximation scheme described above does not yield consistent 
results for Thue--Morse, perhaps a reflection of the exotic behavior 
identified in~\cite{LiuQu}.
As an alternative, in Figure~\ref{fig:dimBtm} we show estimates of the box-counting 
dimension for Thue--Morse.  Since for any finite $k$ the covers comprise the finite union of
real intervals, $\log(N_{C_k}(\eps))/\log(1/\eps) \to 1$ as $\eps\to0$.  
However, for finite $\eps$ the shape of the curves in Figure~\ref{fig:dimBtm} 
can suggest rough values for $\bcd(\Sigma)$; cf.~\cite{TFV89}.
Figure~\ref{fig:dimBpdub} repeats the experiment for period doubling; 
note that the plot for $\lambda=4$ shows good agreement with the estimate
for $\dim_H(\Sigma)$ seen in Figure~\ref{fig:dimHpd}.  (The Hausdorff and
box counting dimensions need not be equal; considerable effort was devoted
over the years to showing $\dim_H(\Sigma)=\bcd(\Sigma)$ for all $\lambda>0$ 
for the Fibonacci case, with the complete result obtained only 
recently~\cite{DGY14}.)

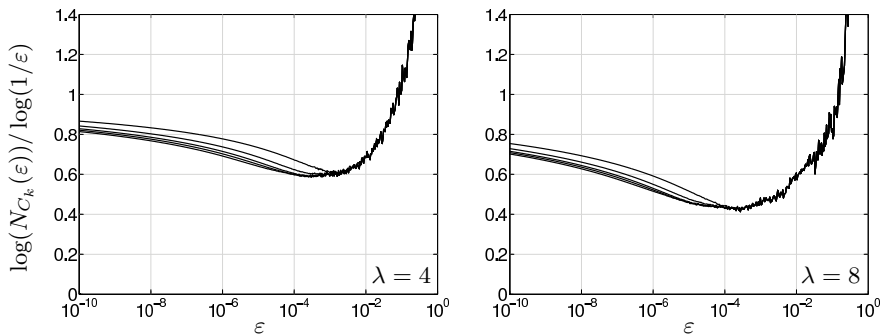
\begin{figure}[t!]
\hfill \includegraphics[scale=0.3]{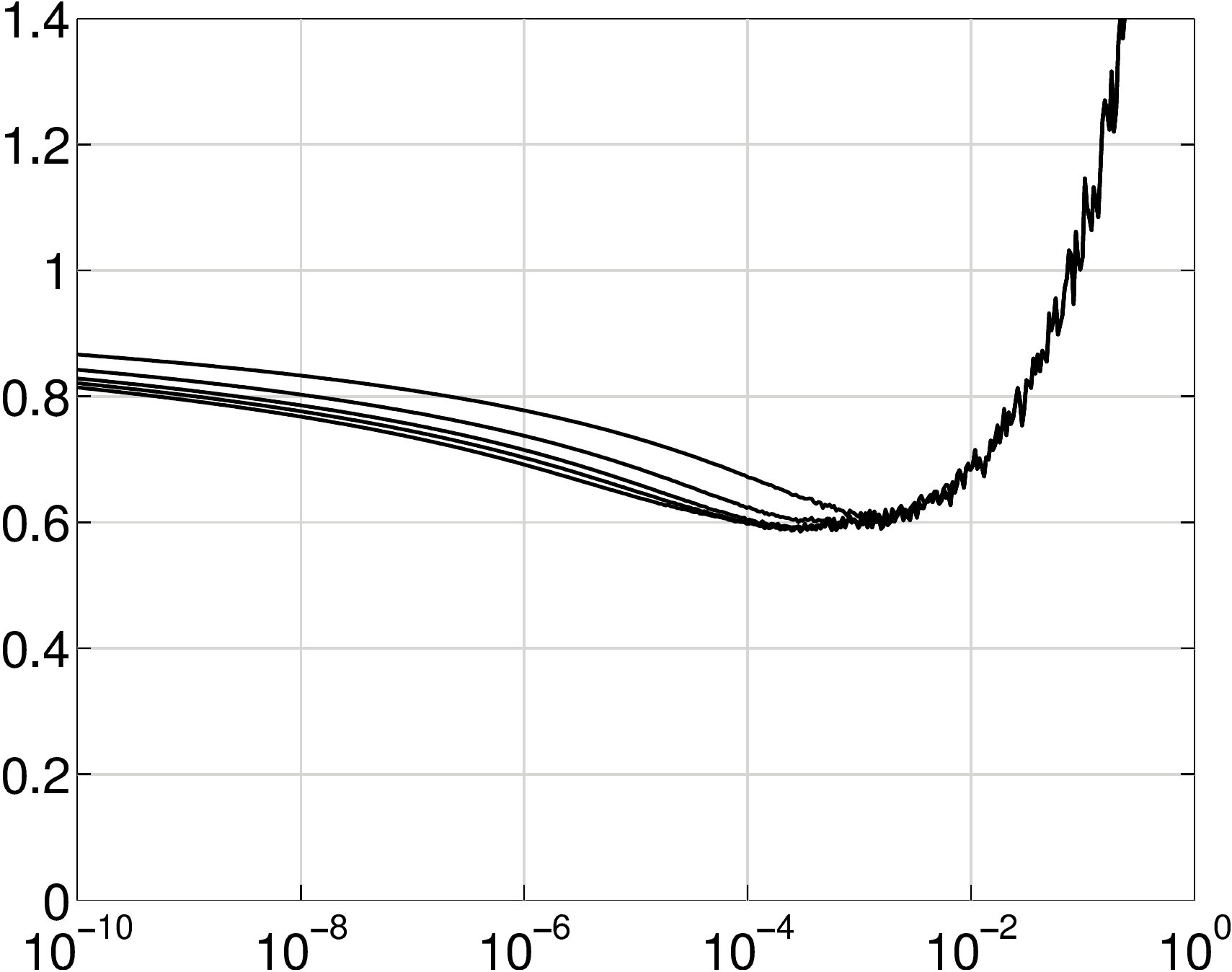} \quad
  \includegraphics[scale=0.3]{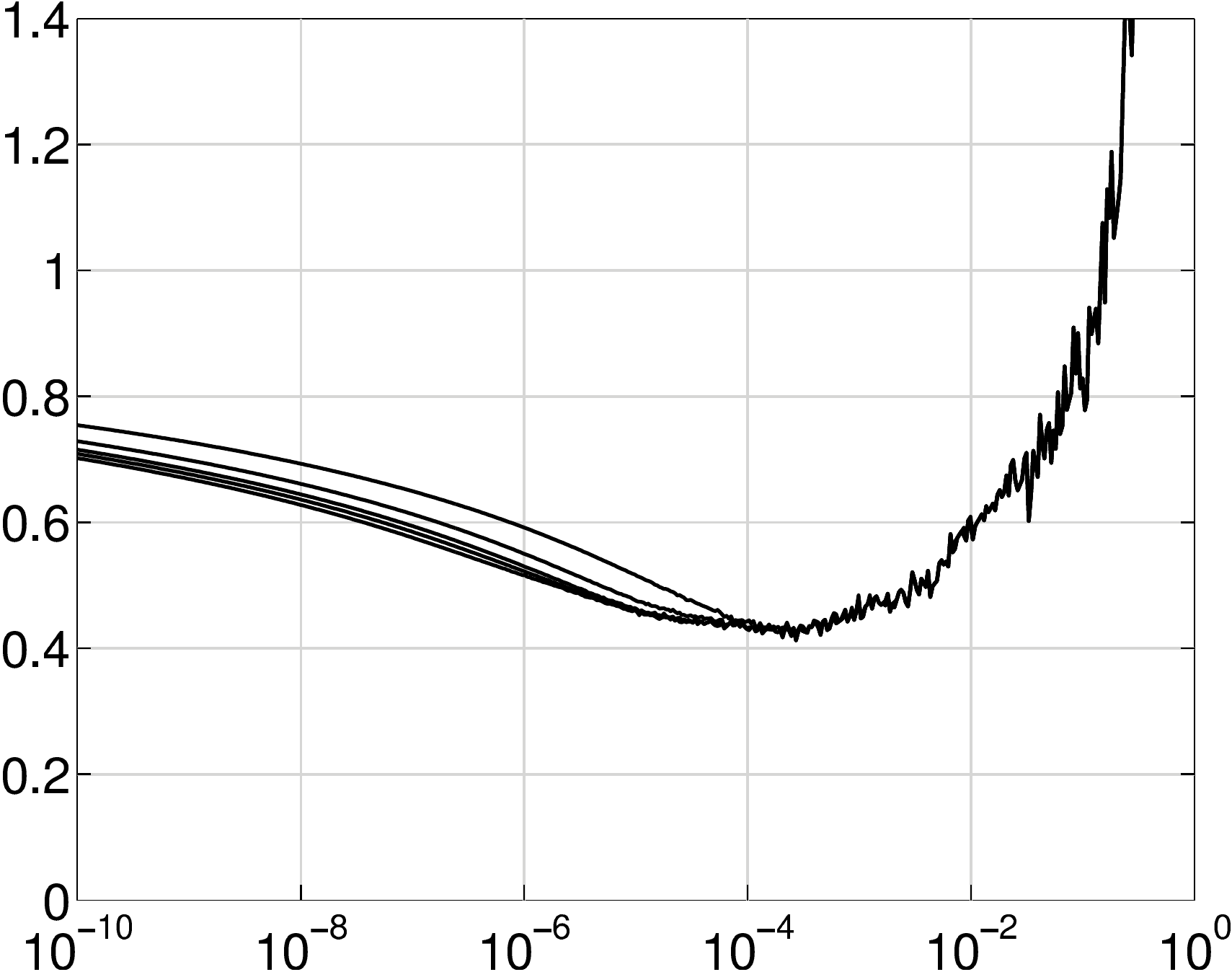}
\begin{picture}(0,0)
\put(-330,17){\rotatebox{90}{\small $\log(N_{C_k}(\eps))/\log(1/\eps)$}}
\put(-238,-6){\small $\eps$}
\put(-77,-6){\small $\eps$}
\put(-194,13){\small $\lambda=4$}
\put(-33,13){\small $\lambda=8$}
\end{picture}
  \caption{\label{fig:dimBtm} 
Estimates of $\bcd(\Sigma)$ for the Thue--Morse operator
for $\lambda = 4, 8$, using covers $C_k$ with $k=5,\ldots, 9$.}  
\end{figure}

\begin{figure}[t!]
\hfill \includegraphics[scale=0.3]{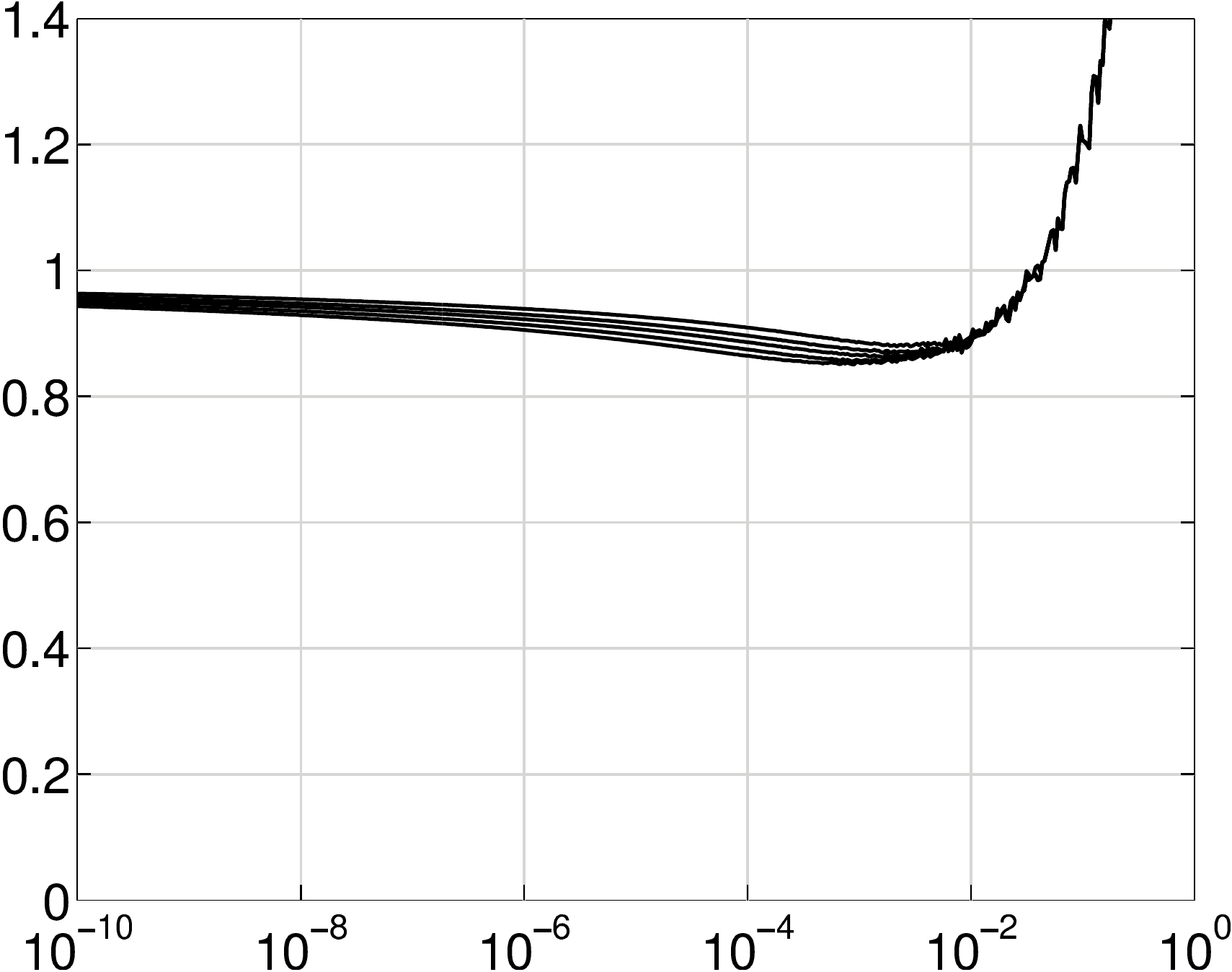} \quad
  \includegraphics[scale=0.3]{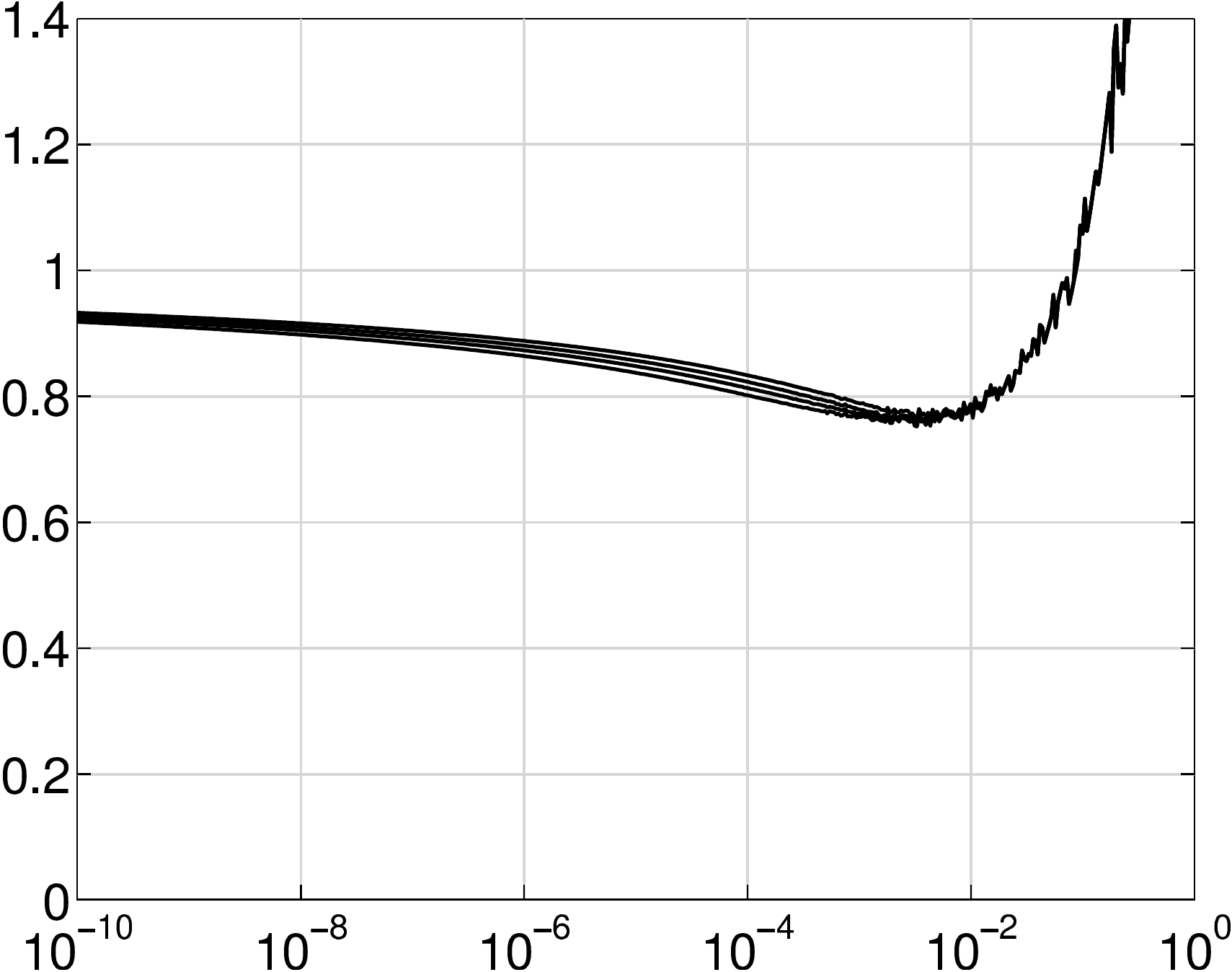}
\begin{picture}(0,0)
\put(-330,17){\rotatebox{90}{\small $\log(N_{C_k}(\eps))/\log(1/\eps)$}}
\put(-238,-6){\small $\eps$}
\put(-77,-6){\small $\eps$}
\put(-194,13){\small $\lambda=4$}
\put(-33,13){\small $\lambda=8$}
\end{picture}
  \caption{\label{fig:dimBpdub} 
Estimates of $\bcd(\Sigma)$ for the period-doubling operator
for $\lambda = 4, 8$, using covers $C_k$ with $k=5,\ldots, 9$.}  
\end{figure}
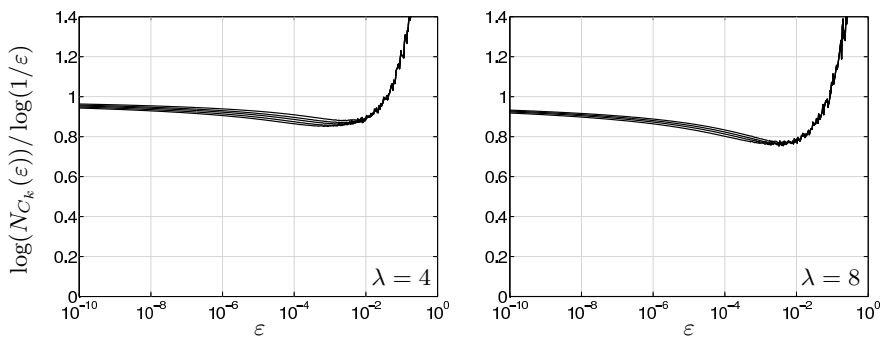

The spectral covers we have described can behave in rather subtle ways.
To illustrate this fact, Figure~\ref{fig:tm_largestgap} shows the largest gap 
in the Thue--Morse cover $\Sigma_k^\a \cup \Sigma_{k+1}^\a$ as a function of
the parameter~$\lambda$.\ \ 
Bellisard~\cite{Bel90} showed that, for the aperiodic model, 
this gap should behave like $\lambda^{\log 4/\log 3}$ as $\lambda \to 0$.  
The covers satisfy this characterization for intermediate values of $\lambda$,
but appear to behave instead like $\lambda^{2}$ as $\lambda\to0$.\ \ 
The larger the value of $k$ (hence the longer the periodic approximations), 
the larger the range of $\lambda$ values that are consistent with Bellisard's spectral asymptotics. 
This scenario provides another justification for the study of large-$k$ approximations:
for this potential, one must consider large $k$ to approximate the spectrum adequately for small~$\lambda$.

Our algorithm also expedites study of the square Hamiltonians $\CJ_{\rm sq}$
acting on $\ell^2(\Z)\times \ell^2(\Z)$ via
\[ (\CJ_{\rm sq} \psi)_{m,n} = \psi_{m-1,n} +  \psi_{m,n-1} + (b_m + b_n) \psi_{m,n} + \psi_{m+1,n} + \psi_{m,n+1}.\]
The spectrum of this operator is $\Sigma+\Sigma$, 
where $\Sigma$ is the spectrum of the standard 1-dimensional
operator with potential $\{b_n\}$.
When $\Sigma$ is a Cantor set, $\Sigma+\Sigma$ could be an interval,
a union of intervals, a \emph{Cantorval}, or a Cantor set; 
see, e.g.,~\cite{MO94,MMR00,PT93a}.\ \ 
For further details and examples for the Fibonacci case, see~\cite{DEG12} and
the computations, based on $\sigma(\BJ_\pm)$, of Even-Dar Mandel and Lifshitz~\cite{EL08}.\ \ 
For the period doubling and Thue--Morse potentials, 
Figure~\ref{fig:follicle2d} illustrates how the covers of the
spectra for these square Hamiltonians develop in $\lambda$,
revealing values of $\lambda$ for which 
$\Sigma+\Sigma$ cannot comprise a single interval.

\begin{figure}[t!]
\begin{center}
\includegraphics[scale=0.45]{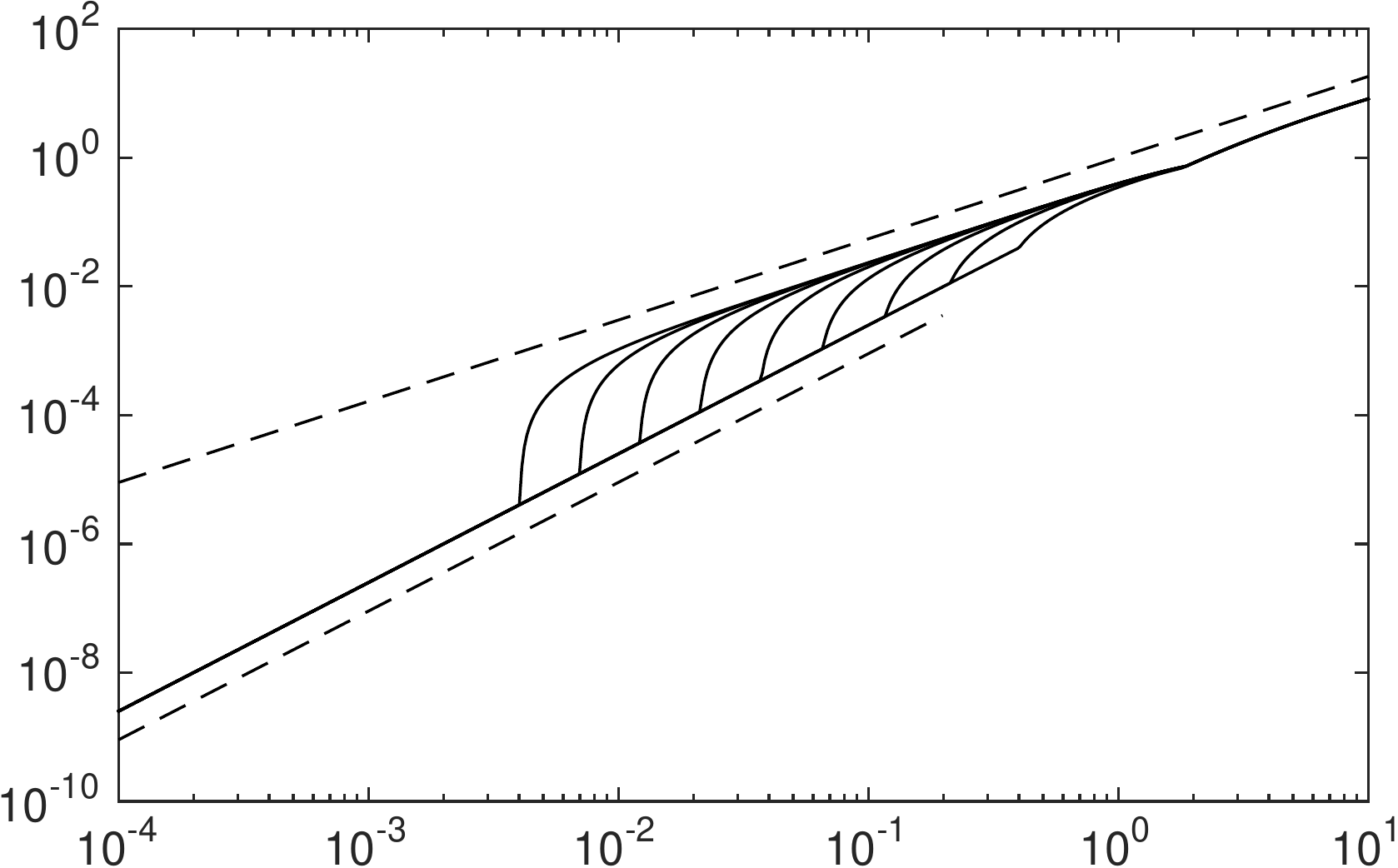}
\begin{picture}(0,0)
\put(-108,-7){\small $\lambda$}
\put(-234,20){\small \rotatebox{90}{largest gap in $\Sigma_k^\a \cup \Sigma_{k+1}^\a$}}
\put(-61,91){\rotatebox{24}{\footnotesize$k=4$}}
\put(-167,50){\rotatebox{24}{\footnotesize $k=12$}}
\put(-197,15){\rotatebox{24}{\footnotesize $\sim \! \lambda^2$}}
\put(-197,66){\rotatebox{17}{\footnotesize $\sim \! \lambda^{\log 4/\log 3}$}}
\end{picture}
\end{center}

\caption{\label{fig:tm_largestgap}
Computed values of the largest gap in the cover $\Sigma_k^\a \cup \Sigma_{k+1}^\a$
of the Thue--Morse spectrum, along with Bellisard's asymptotic description
$\lambda^{\log 4/\log 3}$. 
As $k$ increases, the covers obey the asymptotics for smaller values of $\lambda$.}
\end{figure}
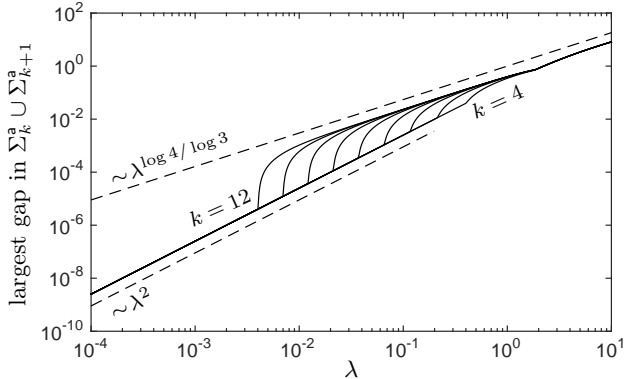

\begin{figure}[t!]
\begin{center}
\includegraphics[scale=0.45]{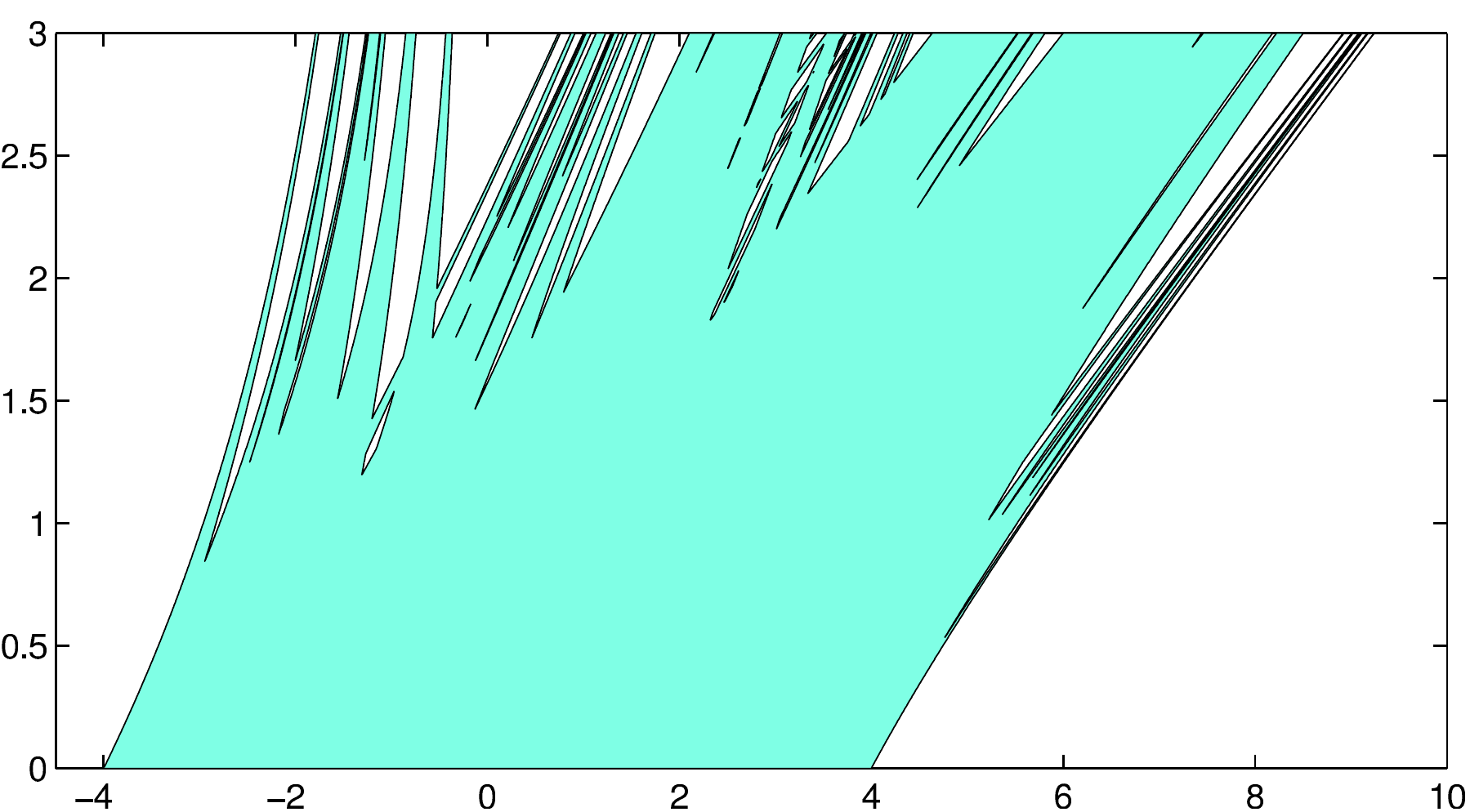}
\begin{picture}(0,0)
\put(-38,23){\small period}
\put(-47,13){\small doubling}
\put(-162,-9){\small $(\Sigma_k^\a \cup \Sigma_k^\b)+(\Sigma_k^\a \cup \Sigma_k^\b)$}
\put(-247,75){\small $\lambda$}
\end{picture}
\\[1.5em]
\includegraphics[scale=0.45]{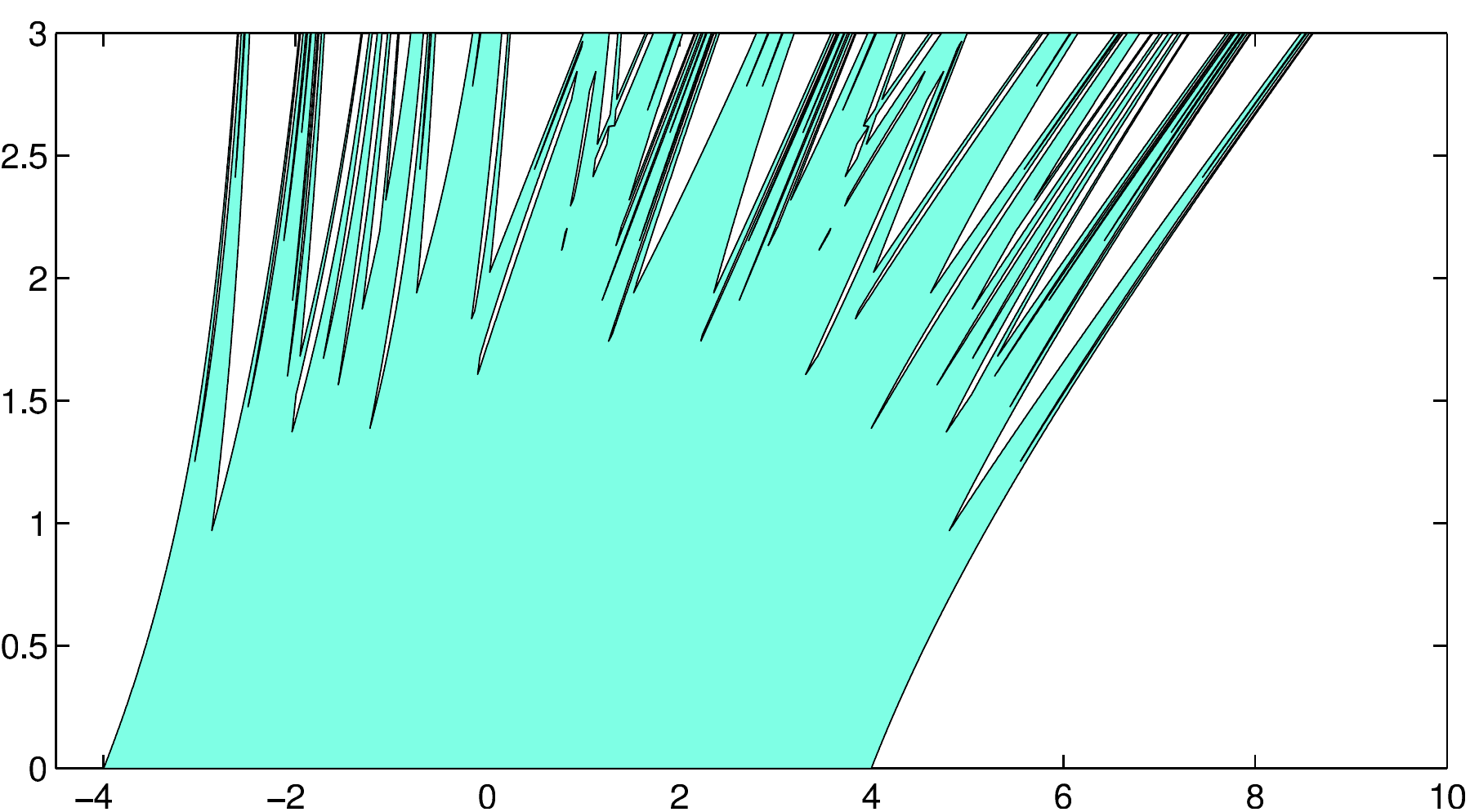}
\begin{picture}(0,0)
\put(-61,13){\small Thue--Morse}
\put(-170,-9){\small $(\Sigma_k^\a \cup \Sigma_{k+1}^\a)+(\Sigma_k^\a \cup \Sigma_{k+1}^\a)$}
\put(-247,75){\small $\lambda$}
\end{picture}
\end{center}

\vspace*{4pt}
\caption{\label{fig:follicle2d}
Spectral covers for the square Hamiltonian for period doubling
and Thue--Morse substitutions, with $k=4$.
}
\end{figure}
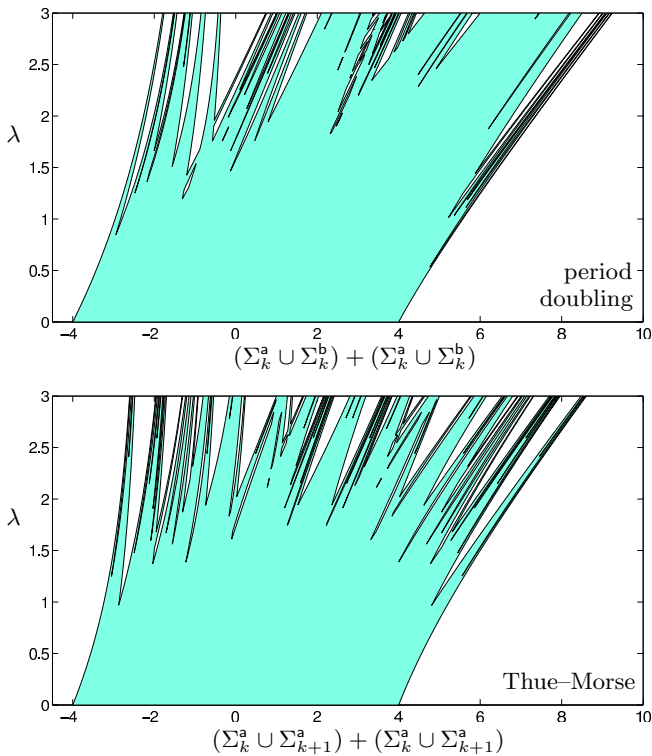

\section{Conclusion}

We have presented a simple $O(K^2)$ algorithm to numerically compute the spectrum of a
period-$K$ Jacobi operator that can be implemented in a few lines of code, then
we used this algorithm to estimate spectral quantities associated with quasiperiodic 
Schr\"odinger operators derived from substitution rules.
The algorithm facilitates study of long-period approximations and extensive parameter studies across a family of operators.
Yet even with the efficient algorithm, 
the accuracy of the numerically computed eigenvalues must be carefully monitored,
or even enhanced with extended precision computations.
This tool can expedite numerical experiments to help 
formulate conjectures about spectral properties 
for a range of quasiperiodic operators.

\subsection*{Acknowledgments}
We thank a referee for helpful comments that led us to study 
the largest Thue--Morse gap (Figure~\ref{fig:tm_largestgap}), 
and David Damanik, Paul Munger, 
Beresford Parlett, and Dan Sorensen for fruitful discussions.
We especially thank May Mei for suggesting computations with 
substitution operators and sharing numerical results, and
Adrian Forster for his Thue--Morse computations.

\bibliographystyle{plain}

\end{document}